\documentclass[twoside,a4paper,12pt,leqno]{amsart}

\usepackage{verbatim,amsxtra,amssymb,amsmath,color,psfrag,graphicx,mathrsfs}

\definecolor{mygrey}{gray}{0.70}
\definecolor{mygreen}{rgb}{0,.75,0}
\definecolor{myred}{rgb}{1,0,0}
\definecolor{orange}{rgb}{1,.5,0}

\numberwithin{equation}{section}

\newtheorem{thm}[equation]{Theorem}

\newtheorem{cor}[equation]{Corollary}
\newtheorem{lem}[equation]{Lemma}
\newtheorem{prop}[equation]{Proposition}

\theoremstyle{definition}
\newtheorem{defn}[equation]{Definition}

\theoremstyle{remark}
\newtheorem{rem}[equation]{Remark}
\newtheorem{ex}[equation]{Example}

\newcommand{\thmref}[1]{Theorem~\ref{#1}}
\newcommand{\propref}[1]{Proposition~\ref{#1}}
\newcommand{\lemref}[1]{Lemma~\ref{#1}}
\newcommand{\corref}[1]{Corollary~\ref{#1}}
\newcommand{\defref}[1]{Definition~\ref{#1}}
\newcommand{\remref}[1]{Remark~\ref{#1}}
\newcommand{\exref}[1]{Example~\ref{#1}}
\newcommand{\figref}[1]{Figure~\ref{#1}}
\newcommand{\secref}[1]{Section~{\bf\ref{#1}}}

\newcommand\A{\mathcal A}
\renewcommand\a{\alpha}

\renewcommand\b{\beta}
\newcommand\bA{{\breve\A}}

\newcommand\bs{\backslash}
\newcommand\bS{\breve\S}

\newcommand\C{\mathcal C}

\newcommand\D{\mathcal D}
\renewcommand\d{\delta}
\newcommand\De{\Delta}
\newcommand\diag{\operatorname{diag}}

\newcommand\e{\varepsilon}
\newcommand\E{\mathcal E}
\newcommand\Edges{\mathsf{Edges}}

\newcommand\f{\varphi}

\newcommand\g{\gamma}
\newcommand\G{\Gamma}

\newcommand\h{\zeta}
\newcommand\HBall{\operatorname{HBall}}
\newcommand\HH{\mathbf H}
\newcommand\Hor{\operatorname{Hor}}

\renewcommand\l{\lambda}
\newcommand\La{\Lambda}
\newcommand\la{\langle}

\newcommand\m{\mathfrak m}

\renewcommand\O{\Omega}
\renewcommand\o{\omega}
\newcommand\ov{\overline}

\newcommand\Paths{\mathsf{Paths}}
\newcommand\pt{\partial}

\renewcommand\r{\rho}
\renewcommand\Re{\mathbb R}
\newcommand\ra{\rangle}

\renewcommand\S{\mathcal S}
\newcommand\s{\sigma}
\newcommand\Si{\Sigma}
\newcommand\supp{\mathop{\text{\rm supp}}}

\newcommand\T{\mathcal T}
\renewcommand\t{\tau}
\newcommand\Th{\Theta}
\renewcommand\th{\theta}

\newcommand\vs{\varsigma}

\newcommand\wh{\widehat}

\newcommand\wt{\widetilde}

\newcommand\X{\mathcal X}

\newcommand\Z{\mathbb Z}

\begin{document}

\phantom{z}

\vskip -2cm

\title[Ergodic properties of boundary actions]{Ergodic properties of
boundary actions and Nielsen--Schreier theory}

\author[R. Grigorchuk]{Rostislav Grigorchuk}
\address{Mathematics Department, Texas A{\&}M University, College Station, TX 77843-3368, USA}
\email{grigorch@math.tamu.edu}

\thanks{The authors acknowledge support of the Fonds National Suisse de la Recherche
Scientifique and of the Swedish Royal Academy of Sciences. The authors are
grateful to the Erwin Schr\"odinger Institute (ESI) in Vienna for the
hospitality and the excellent working conditions during the ``Amenability
2007'' programme. The first author was also supported by the NSF grants
DMS-0456185, DMS-0600975, whereas the second author was partially supported by
funding from the Canada Research Chairs program. The first and the second
authors also express their gratitude to the Institut des Hautes \'Etudes
Scientifiques in Bures-sur-Yvette}

\author[V. Kaimanovich]{Vadim A. Kaimanovich}
\address{Department of Mathematics and Statistics, University of Ottawa, 585 King Edward, Ottawa ON, K1N 6N5, Canada}
\email{vkaimano@uottawa.ca, vadim.kaimanovich@gmail.com}

\author[T. Nagnibeda]{Tatiana Nagnibeda}
\address{Section de math\'ematiques, Universit\'e de Gen\`eve, 2-4, rue du Li\`evre, c.p. 64 1211 Gen\`eve 4,
Switzerland}
\email{tatiana.smirnova-nagnibeda@unige.ch}

\date{\today}

\begin{abstract}
We study the basic ergodic properties (ergodicity and conservativity) of the
action of an arbitrary subgroup $H$ of a free group $F$ on the boundary $\pt
F$ with respect to the uniform measure. Our approach is geometrical and
combinatorial, and it is based on choosing a system of Nielsen--Schreier
generators in $H$ associated with a geodesic spanning tree in the Schreier
graph $X=H\bs F$. We give several (mod 0) equivalent descriptions of the Hopf
decomposition of the boundary into the conservative and the dissipative parts.
Further we relate conservativity and dissipativity of the action with the
growth of the Schreier graph $X$ and of the subgroup $H$ ($\equiv$ cogrowth of
$X$), respectively. We also construct numerous examples illustrating
connections between various relevant notions.
\end{abstract}

\maketitle

\thispagestyle{empty}

\vskip-2cm

{\footnotesize\tableofcontents}

\section*{Introduction}

In 1921 Jacob Nielsen \cite{Nielsen21} proved that any finitely generated
subgroup of a free group is itself a free group. His proof was based on a
rewriting procedure which allows one to reduce an arbitrary finite system of
elements of a free group to a system of free generators. Since then Nielsen's
method has become one of the main tools in the combinatorial group theory
\cite{Magnus-Karrass-Solitar76,Lyndon-Schupp01}. It is used in the study of
the group of automorphisms of a free group, for solving equations in free
groups and in numerous other applications. Its scope is by no means restricted
to free groups and extends to the combinatorial group theory at large,
$K$-theory and topology.

In 1927 Nielsen's result was extended by Otto Schreier \cite{Schreier27} to
arbitrary subgroups in another seminal work (where, in particular, what is
currently known as Schreier graphs was introduced). Under the name of the
Nielsen--Schreier theorem it is now one of the bases of the theory of infinite
groups. Schreier's method is at first glance quite different from Nielsen's
and uses families of coset representatives (transversals). That Nielsen and
Schreier actually arrived at essentially the same generating systems became
clear much later and was proved in \cite{Hall-Rado48,Karrass-Solitar58}.

In this work we show that the Nielsen--Schreier theory is useful in the
ergodic theory, and our main result is its application to the study of the
ergodic properties of the boundary action of arbitrary subgroups of a finitely
generated free group. On the other hand, our point of view ``from infinity''
(based on using dynamical invariants of the boundary action) sheds new light
on geometry of Schreier graphs and associated subgroups.

\vskip5mm

The boundary theory occupies an important place in various mathematical
fields: geometric group theory, rigidity theory, theory of Kleinian groups,
potential analysis, Markov chains, to name just a few. The free group is one
of the central objects in the study of boundaries of groups. Its simple
combinatorial structure makes of it a convenient test-case which contributes
to the understanding of general concepts, both in the group-theoretic (as the
free group is the universal object in the category of discrete groups) and
geometric (as its Cayley graph, the homogeneous tree, is a discrete analogue
of the constant curvature hyperbolic space) frameworks.

There exist many different boundaries of a group corresponding to various
compactifications: the space of ends, the Martin boundary, the visual
boundary, the Busemann boundary, the Floyd boundary, etc. There is also a
measure-theoretical notion of the Poisson(--Furstenberg) boundary, which is
the one especially important for the present study. In the case of the free
group $F$ freely generated by a finite set $\A$, all these notions coincide,
and the \emph{boundary} $\pt F$ can be realized as the space $\bA_r^\infty$ of
infinite freely reduced words in the alphabet $\bA=\A\cup \A^{-1}$. The action
of the group on itself extends by continuity to a continuous action on $\pt
F$.

The choice of the generating set $\A$ determines a natural \emph{uniform}
probability measure $\m$ on $\pt F$ which is quasi-invariant under the action
of $F$. This measure can also be interpreted in a number of other ways.
Namely, as the measure of maximal entropy of the unilateral Markov shift in
the space of infinite irreducible words, as a conformal density (Patterson
measure), or as the hitting ($\equiv$ harmonic) measure of the simple random
walk on the group. In the latter interpretation the measure space $(\pt F,\m)$
is actually isomorphic to the Poisson boundary of the random walk, and it is
this interpretation that plays an important role in our work.

\medskip

The main goal of the present paper is to study the basic ergodic properties,
i.e., ergodicity and conservativity, of the action of an arbitrary subgroup
$H\le F$ on the boundary $\pt F$ with respect to the measure $\m$. Our
principal results are:

\begin{itemize}
    \item
An explicit combinatorial description of the Hopf decomposition of the
boundary action in terms of what we call the Schreier limit set (\thmref{th:1}
and \thmref{th:2});
    \item
Identification of the conservative part of the Hopf decomposition with the
horospheric limit set (\thmref{th:coinc});
    \item
A sufficient condition of complete dissipativity of the boundary action
(\thmref{th:4}) and a necessary and sufficient condition of its conservativity
(\thmref{th:cor3}) in terms of the growth of the group $H$ ($\equiv$ the
cogrowth of the associated Schreier graph $X$) and of $X$, respectively;
    \item
Numerous new examples illustrating and clarifying the interrelations between
various conditions (\secref{sec:exincl} and \secref{sec:exx}).
\end{itemize}

On the other hand, we expect our approach to be useful for purely algebraic
problems as well. For instance, our analysis of the ergodic properties of the
boundary action allows us to give a conceptual proof of an old theorem of
Karrass--Solitar on finitely generated subgroups of a free group
(\remref{rem:KarSol}).

\medskip

Recall that an action of a countable group is called \emph{ergodic} with
respect to a quasi-invariant measure if it has no non-trivial invariant sets.
Any action (on a Lebesgue space) admits a unique \emph{ergodic decomposition}
into its \emph{ergodic components}. An action is called \emph{conservative} if
it admits no non-trivial \emph{wandering set} (i.e, such that its translations
are pairwise disjoint). There is always a maximal wandering set, and the union
of its translations is called the \emph{dissipative part} of the action. Any
action admits the so-called \emph{Hopf decomposition} into the conservative
and dissipative parts. These parts can also be described as the unions of all
the purely non-atomic, and, respectively, of all the atomic ergodic
components. It is important to keep in mind that the Hopf decomposition (as
well as other measure theoretic notions) is defined (mod 0), i.e., up to
measure $0$ subsets.

It is pretty straightforward to see that ergodicity of the boundary action is
equivalent to the \emph{Liouville property} of the simple random walk on the
Schreier graph $X$ (i.e., to the absence of non-constant bounded harmonic
functions on $X$), see \secref{sec:RW}. On the other hand, as was shown by
Kaimanovich \cite{Kaimanovich95}, the boundary action of a non-trivial normal
subgroup is always conservative. In particular, if $G=F/H$ is any
non-Liouville (for example, non-amenable) group, then the action of the normal
subgroup $H$ on $(\pt F,\m)$ is conservative without being ergodic. The only
other previously known example of the Hopf decomposition of a boundary action
was the one of completely dissipative $\Z$-actions \cite{Kaimanovich95}.

\medskip

The starting point of our approach is the \emph{Schreier graph} structure on
the quotient homogeneous space $X=H\bs F$. To quote \cite[Section
2.2.6]{Stillwell93}, Schreier's method ``begs to be interpreted in terms of
spanning trees'' in the Schreier graph. Indeed, there is a one-to-one
correspondence between Schreier generating systems for the subgroup $H$ and
spanning trees in $X$ rooted at the origin $o=H$, which we remind in
\secref{sec:span} and \secref{sec:Sch}. This correspondence consists in
assigning the associated cycle in $X$ to any edge removed when passing to the
spanning tree (\thmref{th:span}).

By interpreting points of the boundary $\pt F$ as infinite paths without
backtracking issued from the origin $o=H$ in the Schreier graph $X$, we define
two subsets of $\pt F$: the \emph{Schreier limit set} $\O$ and the
\emph{Schreier fundamental domain} $\De$ (\defref{def:od}). The set $\O$
corresponds to the paths which pass infinitely many times through
$\Edges(X)\setminus\Edges(T)$ and is homeomorphic to the set $\pt H$ of
infinite irreducible words in the alphabet $\bS=S\sqcup S^{-1}$, whereas the
set $\De$ corresponds to the rays issued from the origin in the tree $T$, and
is homeomorphic to the boundary $\pt T$ of $T$. These sets give rise to a
decomposition $\pt F = \left(\bigsqcup_{h\in H}h \De\right) \sqcup \O$
(\thmref{th:1}).

However, in order to study this decomposition further we have to impose an
additional condition on the Schreier generating system by requiring it to be
\emph{minimal}, which means that the corresponding spanning tree $T$ is
geodesic (the class of minimal Schreier systems coincides with the class of
\emph{Nielsen} generating systems, see \secref{sec:minSch} and the references
therein). Under this assumption we prove that the above decomposition, indeed,
coincides (mod 0) with the Hopf decomposition of the boundary action
(\thmref{th:2}).

The topological counterpart of the Hopf decomposition of the boundary action
is the decomposition of the boundary $\pt F$ into a union of the closed
$H$-invariant \emph{limit set} $\La=\La_H$ (the closure of $H$ in the
compactification $\wh F= F \cup \pt F$) and its complement $\La^c$. According
to a general result (valid for all Gromov hyperbolic spaces), the restriction
of the $H$-action to $\La$ is \emph{minimal} (any orbit is dense), whereas its
restriction to $\La^c$ is \emph{properly discontinuous} (no orbit has
accumulation points). The decomposition $\pt F=\La\sqcup\La^c$ corresponds to
the decomposition of the Schreier graph $X$ into a union of its \emph{core}
$X_*$ and the collection of \emph{hanging branches} (\thmref{th:full}; see
\secref{sec:hang} for the definitions). In particular, $\La=\pt F$ if and only
if $X$ has no hanging branches.

The Schreier limit set $\O$ is contained in the full limit set $\La$, which
corresponds to the fact that proper discontinuity of the boundary action on
$\La^c$ implies its complete dissipativity with respect to any quasi-invariant
measure (in particular, the uniform measure~$\m$). Geometrically, any hanging
branch in $X$ gives rise to a non-trivial wandering set in $\La^c$. However,
the action on $\La$ may also have a non-trivial dissipative part, or even be
completely dissipative. For instance, it may so happen that the Schreier graph
$X$ has no hanging branches at all (i.e., $\La=\pt F$), but nonetheless the
boundary action is completely dissipative (\exref{ex:mingrowthdiss}).

We introduce the \emph{small} (resp., \emph{big}) \emph{horospheric limit set}
$\La^{horS}=\La^{horS}_H$ (resp., $\La^{horB}=\La^{horB}_H$) of the subgroup
$H$ as the set of all the points $\o\in\La$ such that any (resp., a certain)
horoball centered at $\o$ contains infinitely many points from $H$, and show
that the Schreier limit set $\O$ is sandwiched between $\La^{horS}$ and
$\La^{horB}$, but coincides with them (mod~0) with respect to the measure $\m$
(\thmref{th:3} and \thmref{th:coinc}). We also establish certain other
inclusions and show by appropriate examples that all of them are strict
(\secref{sec:exincl}).

If the subgroup $H$ is finitely generated (i.e, if the core $X_*$ is finite),
then the \emph{Hopf alternative} between conservativity and complete
dissipativity holds: either the Schreier graph $X$ is finite and the boundary
action of $H$ is ergodic (therefore, conservative), or $X$ is infinite and the
boundary action is completely dissipative (\thmref{th:pr7}). However, for
infinitely generated subgroups the relationship between the ergodic properties
of the boundary action and the geometry of the Schreier graph $X$ is much more
complicated (as illustrated by numerous examples in \secref{sec:exx}).

We prove that if the exponential growth rate of $H$ ($\equiv$ the
\emph{cogrowth} of~$X$) satisfies the inequality $v_H<\sqrt{2m-1}$, where $m$
is the number of generators of $F$ (i.e., if $v_H<\sqrt{v_F}$), then the
boundary action of $H$ is completely dissipative (\thmref{th:4}). On the other
hand, we show (\thmref{th:cor3}) that the boundary action of $H$ is
conservative if and only if $\lim_n |S_X^n|/|S_F^n| = 0$, where $S_X^n$
(resp., $S_F^n$) is the radius $n$ sphere in $X$ (resp., $F$) centered at the
origin $o$ (resp., at the identity $e$). In particular, if the exponential
growth rate $v_X$ of the Schreier graph $X$ satisfies the inequality
$v_X<2m-1$, then the boundary action is conservative (\corref{cor:pr5}).

\medskip

Markov chains (not only the aforementioned simple random walks, but also the
other chains described in \secref{sec:asso-chains}) play an important role in
understanding the ergodic properties of the boundary action. Another
measure-theoretical tool which we use in this paper is the relationship of the
boundary action with two other natural actions of the subgroup~$H$ (see
\secref{sec:ergo} for references and more details).

The first one is the action on the square $\pt^2 F$ of the boundary $\pt F$
endowed with the square of the uniform measure $\m$. The ergodic properties of
this action are the same as for the (discrete) \emph{geodesic flow} on the
Schreier graph $X$ and are described by the \emph{classical Hopf alternative}
(\thmref{th:HTS}): the action of $H$ on $(\pt^2 F,\m^2)$ is either ergodic
(therefore, conservative) or completely dissipative. Moreover, ergodicity of
this action is equivalent to divergence of the \emph{Poincar\'e series}
$\sum_{h\in H} (2m-1)^{-|h|}$. Note that the ergodic behaviour of the action
on $\pt F$ is much more complicated than of that on $\pt^2 F$: for instance,
the Hopf alternative for the former, generally speaking, holds only in the
finitely generated case. It is interesting that one of our descriptions of the
Hopf decomposition of the action on $\pt F$ deals with a series similar to the
Poincar\'e series arising for the action on $\pt^2 F$. However, once again, it
is more complicated as it involves the \emph{Busemann functions} rather than
plain distances (see \thmref{th:2} and \thmref{th:coinc}).

The second auxiliary action is the action of the group $H$ on the space of
horospheres in $F$, i.e., the $\Z$-extension of the action on $\pt F$
determined by the \emph{Busemann cocycle}. Geometrically, this action
corresponds to what could be called (by analogy with Fuchsian groups)
\emph{``horocycle flow''} on the Schreier graph. We use the fact that (unlike
for the action on $\pt^2 F$) the ergodic properties of this action are
precisely the same as for the original action on $\pt F$
(\thmref{th:horocycle}).

\medskip

It is a commonplace that the homogeneous tree is a ``rough sketch'' of the
hyperbolic plane. Both these spaces are Gromov hyperbolic (even CAT($-1$)),
and their isometry groups are ``large enough'' (so that the rotations around
any reference point inside act transitively on the hyperbolic boundary). The
subgroups of the free group $F$, which are the object of our consideration,
are just the torsion free discrete groups of isometries of the Cayley tree of
$F$. Thus, the question about analogous results for discrete isometry groups
in the hyperbolic setup --- be it for the usual hyperbolic plane (Fuchsian
groups), higher dimensional simply connected spaces of constant negative
curvature (Kleinian groups), arbitrary non-compact rank~1 symmetric spaces,
general CAT($-1$) or Gromov hyperbolic spaces, even for spaces which are
hyperbolic in a weaker form --- cannot fail to be asked.

The work on the present article prompted the second author to show that the
identification of the conservative part of the boundary action with the big
horospheric limit set $\La^{horB}$ is actually valid in the full generality of
a discrete group of isometries of an arbitrary Gromov hyperbolic space endowed
with a quasi-conformal boundary measure \cite{Kaimanovich10} (see the
references therein for a list of earlier particular cases of this result). The
proof uses the fact that, by definition, the logarithms of the Radon--Nikodym
derivatives of such a measure are (almost) proportional to the Busemann
cocycle, in combination with the description of the Hopf decomposition of an
arbitrary action in terms of the orbitwise sums of the Radon--Nikodym
derivatives (\thmref{thp:3}). However, this is the only situation in our paper
when the cases of the free group and of the hyperbolic plane are
specializations of a common general result. Even here we obtain, in terms of
Nielsen--Schreier generators, a much more detailed information about the Hopf
decomposition than in the general case (\thmref{th:coinc}).

Two other occasions when our results have analogues for Fuchsian or Kleinian
groups are \thmref{th:4} and \thmref{th:cor3} from
\secref{sec:Sch-graph-geometry} which give qualitative criteria of complete
dissipativity and conservativity of the boundary action, respectively. Here
common general results are unknown, and our methods are completely different
from those used in the hyperbolic situation by Patterson \cite{Patterson77}
and Matsuzaki \cite{Matsuzaki05} in the first case (see \remref{rem:qq}) and
by Sullivan \cite{Sullivan81} in the second case (see \remref{rem:s2}).
Although the ``hyperbolic'' techniques most likely might be carried over to
our situation as well, our approach is much more appropriate in the discrete
case as it uses combinatorial tools not readily available in the continuous
case. For instance, we obtain \thmref{th:cor3} as a corollary of \thmref{th:5}
which gives an explicit formula for the measure of a certain canonical
wandering set; a hyperbolic analogue of \thmref{th:5} is unknown.

\medskip

Let us finally mention some open questions arising in connection with the
present work. The most obvious one is to what extent our results can be
carried over to other boundary measures. The first candidate would be the
conformal (Patterson) measures which are singular with respect to the uniform
measure $\m$ in the case when the growth $v_H$ of $H$ is strictly smaller than
the growth of the ambient group $F$. By a general result from
\cite{Kaimanovich10}, in this case the conservative part can still be
identified with the big horospheric limit set (see above), but we do not know
to what extent the combinatorial machinery developed in the present paper can
be adapted to this situation. Of course, one can also try to generalize our
technique to the nearest relatives of free groups, i.e., to word hyperbolic
groups, or even to general discrete groups of isometries of Gromov hyperbolic
spaces.

In a different direction, it would be interesting to investigate the
properties described in the present paper for \emph{random Schreier graphs}
determined by a probability measure invariant with respect to the ``root
moving'' equivalence relation (in other words, a conjugation invariant
probability measure on the space of subgroups of $F$, see
\cite{Vershik10,Abert-Glasner-Virag11p}).

Finally, a more concrete question concerns existence of conservative boundary
actions (with respect to the uniform measure) with $v_H=\sqrt{2m-1}$, cf.
Theorem 4.2. The analogous question is also open for Fuchsian groups, see
\remref{rem:q} and \remref{rem:qq}. For the free group this situation is
especially intriguing because from the spectral point of view Schreier graphs
with $v_H\le\sqrt{2m-1}$ are precisely the \emph{infinite Ramanujan graphs}
(i.e., have the minimal possible spectral radius, see formula (4.1)). Our
Theorem 4.2 implies, in the case when $v_H<\sqrt{2m-1}$, some properties
conjectured to be true for all infinite Ramanujan graphs: absence of the
Liouville property \cite[Conjecture 1]{Benjamini-Kozma10p} and the fact that
the random walk neighbourhood sampling along the graph converges to the
regular tree \cite[Question 11]{Abert10p}.

\section{Nielsen--Schreier theory and the boundary action}
\label{sec:Nielsen-method}

Let $F$ denote the free group freely generated by a finite set $\A$ with
$|\A|=m\ge 2$, and let $\bA=\A\sqcup\A^{-1}$. The Cayley graph $\G(F,\bA)$ is
a homogeneous tree of degree $|\bA|=2|\A|$.

We shall use the notation $\bA^*$ (resp.,~$\bA^\infty$) for the sets of all
finite (resp., right infinite) words in the alphabet $\bA$. The length of a
word $w$ is denoted by $|w|$. For the subsets of $\bA^*$ and $\bA^\infty$
consisting of freely reduced words we shall add the subscript $r$, so that
there is a canonical map
\begin{equation} \label{eq:sigma}
\s:F\to\bA^*_r
\end{equation}
identifying $F$ and $\bA^*_r$.

Any subgroup $H\le F$ of a free group $F$ is also free. It was proved by
Nielsen \cite{Nielsen21} (for finitely generated subgroup) and Schreier
\cite{Schreier27} by giving two different constructions of free generating
sets in $H$, see \cite{Magnus-Karrass-Solitar76} and the references therein.
As it turned out, Nielsen's generating systems are just a particular case of
Schreier's systems (see \thmref{thm:lem1} below). We shall begin by recasting
the original symbolic construction of Schreier (described in \cite[Section
2.3]{Magnus-Karrass-Solitar76}) in terms of spanning trees in the Schreier
graph $X\cong H\bs F$ (cf. \cite[Section~2.2.6]{Stillwell93} and
\cite[Section~6]{Kapovich-Myasnikov02}). Further we shall construct a
decomposition of the boundary $\pt F$ naturally associated with such a
spanning tree (\thmref{th:1}), which is the main goal of this Section.

\subsection{Spanning trees and Schreier transversals} \label{sec:span}

Given a subgroup $H\le F$, denote by $\G(X,\bA)$ the \emph{Schreier graph} of
the homogeneous space $X=H\bs F$ with respect to $\bA$, i.e., two cosets
$Hg_1,Hg_2\in X$ are connected with an edge if and only if
$g_1^{-1}g_2\in\bA$, in which case the oriented edge $[Hg_1,Hg_2]$ is labelled
with $g_1^{-1}g_2$. Notice that, unlike the Cayley graph, the Schreier graph
may have multiple edges with the same endpoints (but different labels). The
extreme example is $H=F$, when $X$ consists of just one vertex with $|\A|$
attached loops. In the sequel we shall always assume that $X$ is endowed with
the Schreier graph structure.

The Schreier graph $X$ has a distinguished vertex $o=H$, it is connected,
$|\bA|$-regular (loops attached to points $x\in X$ with $xg=g$ for certain
$g\in\bA$ are also counted!), and the set of edge labels around each vertex is
precisely $\bA$. Moreover, the labels assigned to two different orientations
of any edge are mutually inverse. Conversely, any graph with the above
properties is the Schreier graph associated with a subgroup of $F$.

\begin{rem}
By a theorem of Gross any regular graph of even degree can be realized as the
Schreier graph associated to a subgroup of a free group (i.e., its edges can
be labelled in the aforementioned way). It is explained in \cite{Lubotzky95}
for finite graphs; an inductive argument can be used to carry the proof over
to infinite graphs (also see \cite{delaHarpe00}).
\end{rem}

It will be convenient to use the following geometric interpretation of the set
of irreducible words in the alphabet $\bA$:

\begin{prop} \label{pr:1}
The set $\bA^*_r\cong F$ is in one-to-one correspondence
$g\leftrightarrow\pi(g)$ with the set $\Paths_o(X)$ of finite paths without
backtracking in the Schreier graph $X$, starting from the origin $o=H$. This
correspondence amounts to consecutive reading of the edge labels along the
path, starting from the origin.
\end{prop}

Consider a \emph{spanning tree} $T$ in $X$ rooted at the point $o$, so that
the origin $o$ can be connected with any vertex $x\in X$ by a unique path
$[o,x]=[o,x]_T$ which only uses the edges from $T$ (such a tree can easily be
constructed for any connected graph, e.g., see
\cite[Section~2.1.5]{Stillwell93}). Then the set of words associated to these
paths as $x$ runs through the whole set $X$ (see \propref{pr:1}), is a
collection $\T$ of coset representatives (a \emph{transversal}) for the group
$H$. The transversal $\T$ has the property that any initial segment of an
element of $\T$ is itself an element of $\T$. Such transversals are said to
satisfy the \emph{Schreier property}. Conversely, any Schreier transversal
obviously determines a spanning tree in $X$.

\subsection{Schreier generating systems} \label{sec:Sch}

Any Schreier transversal $\T$ (equivalently, the associated spanning tree $T$)
gives rise to a system of free generators for $H$ parameterized by the edges
of $X$ which are not in $T$. Indeed, any such edge
$\E=[x,y]\in\Edges(X)\setminus\Edges(T)$ determines a non-trivial cycle
$\vs_\E=[o,x]_T [x,y] [y,o]_T$ in $X$ obtained by joining the endpoints $x,y$
with $o$ in $T$ by unique paths $[o,x]_T$ and $[y,o]_T$, respectively, see
\figref{fig:cycle}. The corresponding generator $s=s_\E=\pi^{-1}(\vs_\E)$ is
presented by the word $\s(s)$ which consists of the edge labels read along the
path $\vs_\E$ (see \propref{pr:1}). We shall denote by
$\s_-(s),\s_0(s),\s_+(s)\in\bA$ the words which correspond to the parts
$[o,x]_T,[x,y],[y,o]_T$ of the path $\vs_\E$, respectively, so that
\begin{equation} \label{eq:dec}
\s(s) = \s_-(s) \s_0(s) \s_+(s) \;.
\end{equation}
In particular,
\begin{equation} \label{eq:s-+}
|\s_-(s)|=|x|_T\;, \qquad |\s_+(s)|=|y|_T \;,
\end{equation}
where $|\cdot|_T=d_T(o,\cdot)$ is the graph distance from the origin $o$ in
the tree $T$. We denote by $\bS$ the set of all the generators $s=s_\E$ of $H$
obtained in this way, and by $\S\subset\bS$ the set of generators which
correspond to those edges $\E$ which are labelled with elements of $\A$, i.e.,
for which $\s_0(s)\in\A$. Two different orientations of the same edge give a
pair of mutually inverse generators, so that $\bS=\S\sqcup\S^{-1}$.

\begin{figure}[h]
\begin{center}
     \psfrag{o}[r][r]{$o=H$}
     \psfrag{x}[lb][lb]{$x=Hg$}
     \psfrag{y}[lt][lt]{$y=Hg a$}
     \psfrag{a}[l][l]{$a=\s_0(s)$}
     \psfrag{s}[][]{$\vs_\E$}
     \psfrag{e}[l][l]{$\E$}
          \includegraphics[scale=.75]{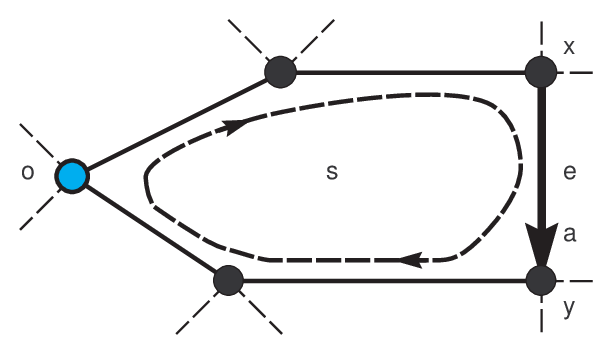}
          \end{center}
          \caption{}
          \label{fig:cycle}
\end{figure}

Obviously, any cycle $\vs$ in $H$ issued from $o$ can be presented as a
composition of the cycles $\vs_\E$ (which correspond to the sequence of edges
from $\Edges(X)\setminus\Edges(T)$ through which $\vs$ passes), so that $\S$
is a generating system for $H$. The fact that it generates $H$ freely follows
from a general theorem of Schreier \cite[Theorem
2.9]{Magnus-Karrass-Solitar76}. In our case, however, there is a more explicit
argument.

\begin{lem} \label{lem:dec}
For any two elements $s,s'\in\bS$ with $s'\neq s^{-1}$ denote by
$$
\a(s,s')=\ov{\s_+(s)\s_-(s')}\in\bA^*_r
$$
the result of the free reduction of the concatenation of the components
$\s_+(s)$ and $\s_-(s')$ from the decomposition \eqref{eq:dec}. Then for any
$h=s_1s_2\dots s_n\in\bS^*_r\cong H$ one has
\begin{equation} \label{eq:decprod}
\s(h) = \s_-(s_1) \s_0(s_1) \a(s_1,s_2) \s_0(s_2) \dots \a(s_{n-1},s_n)
\s_0(s_n) \s_+(s_n) \;.
\end{equation}
\end{lem}

\begin{proof}
Look at the decompositions $\s(s_i)=\s_-(s_i)\s_0(s_i)\s_+(s_i)$
\eqref{eq:dec} for all $i=1,2,\dots,n$. The word $\s_-(s_1)\s_0(s_1)$ ends
with the letter $\s_0(s_1)$ which corresponds to passing through the edge
$\E_{s_1}\in\Edges(X)\setminus\Edges(T)$ associated with the generator $s_1$.
Since no other generator passes through this edge, the letter $\s_0(s_1)$ does
not cancel. In the same way the middle letters $\s_0(s_i)$ do not cancel for
all the other $s_i$.

More geometrically, let $[x,y]=\E_s$ and $[x',y']=\E_{s'}$, then $\a(s,s')$ is
the word obtained be reading the labels along the geodesic segment $[y,x']$
joining the points $y$ and $x'$ in the spanning tree $T$. Thus, letting
$[x_i,y_i]=\E_{s_i}$ for $1\le i\le n$, we have that the path
\eqref{eq:decprod} consists of the consecutive segments
$$
[o,x_1], [x_1,y_1], [y_1,x_2], [x_2,y_2], \dots, [y_{n-1},x_n], [x_n,y_n],
[y_n,o] \;,
$$
see \figref{fig:many} where $n=5$.
\end{proof}

\begin{figure}[h]
\begin{center}
     \psfrag{o}[t][tl]{$o$}
     \psfrag{x1}[r][r]{$x_1$}
     \psfrag{xy}[rb][rt]{$y_1=x_2$}
     \psfrag{y2}[r][r]{$y_2$}
     \psfrag{x3}[l][l]{$x_3$}
     \psfrag{y3}[l][l]{$y_3$}
     \psfrag{x4}[l][l]{$x_4$}
     \psfrag{y4}[l][l]{$y_4$}
     \psfrag{x5}[l][l]{$x_5$}
     \psfrag{y5}[l][l]{$y_5$}
          \includegraphics[scale=.75]{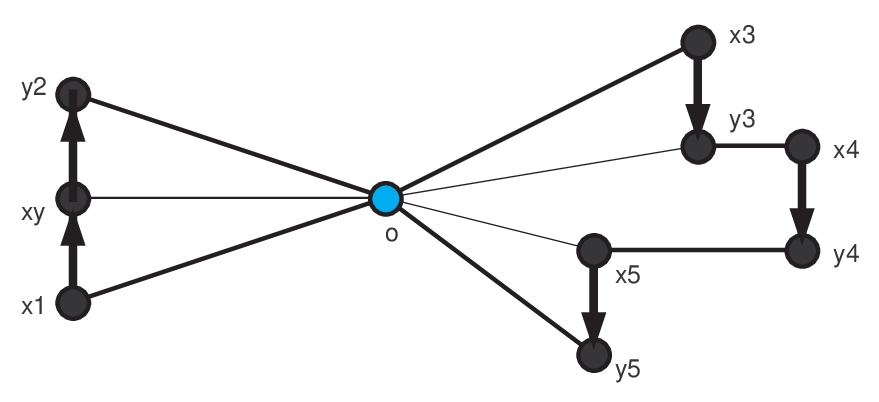}
          \end{center}
          \caption{}
          \label{fig:many}
\end{figure}

We can summarize this discussion in the following way:

\begin{thm} \label{th:span}
Any spanning tree $T$ in the Schreier graph $X$ determines a one-to-one
correspondence $\E\mapsto s_\E,\,s\mapsto\E_s$ between the set of oriented
edges of $X$, which are not in $T$, and the set $\bS=\S\sqcup\S^{-1}$ of the
associated free generators of $H$ and their inverses.
\end{thm}

\begin{defn}
The free generating system $\S$ of the subgroup $H$ is called the
\emph{Schreier system} associated with the spanning tree $T$ (equivalently,
with the corresponding Schreier transversal $\T$).
\end{defn}

\subsection{The boundary map} \label{sec:bdrymap}

There is a natural \emph{compactification} $\wh F= F\cup\pt F$ of the group
$F$. It does not depend on the choice of the generating set $\A$ and admits a
number of interpretations, for instance, as the \emph{end} or as the
\emph{hyperbolic} compactifications of the Cayley graph $\G(F,\bA)$. The
action of the group $F$ on itself extends to a continuous action of $F$ on the
boundary $\pt F$.

In symbolic terms, the map $\s:F\to\bA_r^*$ \eqref{eq:sigma} can be extended
to the \emph{boundary} $\pt F$. This extension (also denoted by $\s$)
$$
\s:\pt F \to \bA^\infty_r
$$
identifies $\pt F$ with the set $\bA^\infty_r$ of infinite freely reduced
words endowed with the product topology of pointwise convergence. A sequence
$g_n\in F$ converges to a boundary point $\o\in\pt F$ if and only if the
finite words $\s(g_n)$ converge to the infinite word $\s(\o)$. The action
$(g,\o)\mapsto g\o$ consists then in concatenation of the associated words
with a subsequent free reduction.

Given a point $\o\in\pt F$ we shall denote by $[\o]_n$ its $n$-th
\emph{truncation}, i.e., the element of $F$ corresponding to the initial
length $n$ segment of the word $\s(\o)$. Geometrically, the sequence
$\left\{[\o]_n\right\}_{n=0}^\infty$ is the \emph{geodesic ray} $[e,\o)$ in
the Cayley graph $\G(F,\bA)$ joining the group identity $e$ with the boundary
point $\o$, see \figref{fig:ray}.

\begin{figure}[h]
\begin{center}
     \psfrag{e}[rb][rb]{$e=[\o]_0$}
     \psfrag{a}[rb][rb]{$[\o]_1$}
     \psfrag{b}[rb][rb]{$[\o]_2$}
     \psfrag{o}[l][l]{$\o$}
          \includegraphics[scale=.75]{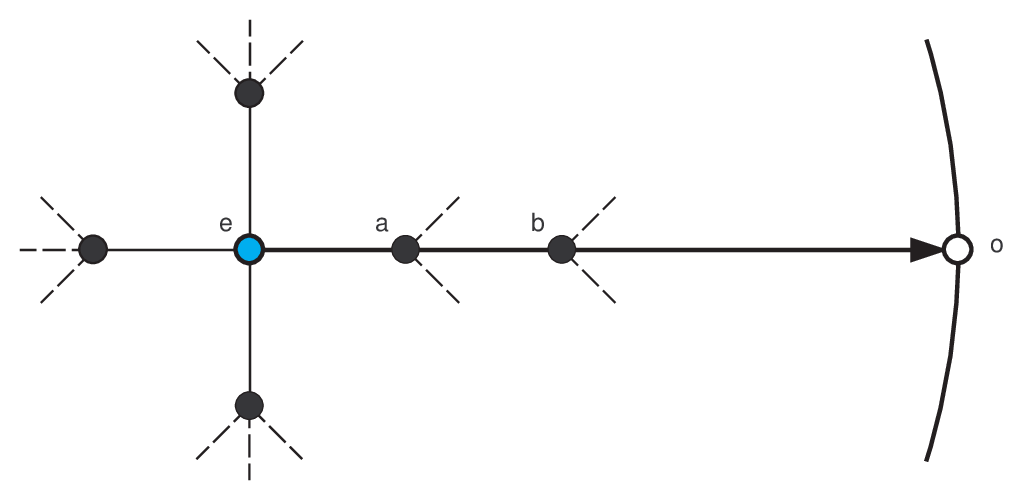}
          \end{center}
          \caption{}
          \label{fig:ray}
\end{figure}

In the same way as for the ambient group $F$, we shall denote by $\pt
H\cong\bS^\infty_r$ the space of infinite freely reduced words in the alphabet
$\bS$ endowed with the product topology of pointwise convergence.

\begin{rem} \label{rem:compct}
The space $\pt H$ is compact if and only if the alphabet $\bS$ is finite,
i.e., the group $H$ is finitely generated.
\end{rem}

\begin{thm} \label{th:bdrymap}
Let $\S$ be the free generating system of a subgroup $H\le F$ determined by a
spanning tree $T$ in the associated Schreier graph $X$ (see \thmref{th:span}).
Then the restriction $\s:H\cong\bS^*_r\to\bA^*_r$ of the map $\s$
\eqref{eq:sigma} extends by continuity to a map $\s^\infty:\pt
H\cong\bS^\infty_r\to\pt F\cong\bA^\infty_r$ as
\begin{equation} \label{eq:lim}
\s^\infty(\xi) = \lim_{n\to\infty} \s([\xi]_n) \;,
\end{equation}
where $\xi=s_1s_2\dots\in\bS^\infty_r\cong\pt H$ is an infinite freely reduced
word in the alphabet $\bS$ and $[\xi]_n=s_1s_2\dots s_n\in\bS^*_r\cong H$ are
its truncations. The extended map $\s^\infty$ is an $H$-equivariant
homeomorphism of $\pt H$ onto its image
\begin{equation} \label{eq:O}
\O=\s^\infty(\pt H)\subset\pt F \;.
\end{equation}
\end{thm}

\begin{proof}
As follows from \lemref{lem:dec}, the limit \eqref{eq:lim} exists, and
\begin{equation} \label{eq:asymp}
\s^\infty(\xi) = \s_-(s_1)\s_0(s_1)\a(s_1,s_2)\s_0(s_2)\a(s_2,s_3)\s_0(s_3)
\dots \;.
\end{equation}
The $H$-equivariance of the limit map $\s^\infty$ is obvious. It order to
check its invertibility note that the initial segments $\s_-(s)\s_0(s)$ of all
the words $\s(s),\,s\in\bS$ are \emph{isolated} in the sense that they do not
occur as initial segments of any other word $\s(s'),\,s'\in \bS$ (this is one
of the defining properties of a \emph{Nielsen system\/}, see below
\secref{sec:minSch}; in our case it directly follows from the construction of
$\bS$). Therefore, the word $\s^\infty(\xi)$ uniquely determines the letter
$s_1\in\S$ such that $\s^\infty(\xi)$ begins with the segment
$\s_-(s_1)\s_0(s_1)$, i.e., the initial letter of $\xi$. By using the
$H$-equivariance and applying the same consideration to $\xi'=s_1^{-1}\xi$ we
recover then the second letter of $\xi$ and so on. Finally, continuity of
$\s^\infty$ follows from formula \eqref{eq:asymp}, whereas continuity of the
inverse map follows from its description in the previous sentence.
\end{proof}

\subsection{A boundary decomposition} \label{sec:bdrydec}

Along with the set $\O$ \eqref{eq:O} we also define a subset $\De\subset\pt
F\cong\bA^\infty_r$ as the set of all the infinite words which do not begin
with any of the segments $\s_-(s)\s_0(s),\,s\in\bS$.

\begin{defn} \label{def:od}
The sets $\O,\De\subset\pt F$ are called the \emph{Schreier limit set} and the
\emph{Schreier fundamental domain}, respectively. They are determined by the
choice of a spanning tree $T$ in the Schreier graph $X$.
\end{defn}

\begin{prop} \label{pr:meas}
The Schreier limit set $\O$ is $G_\d$ in $\pt F$, and the Schreier fundamental
domain $\De$ is closed in $\pt F$.
\end{prop}

\begin{proof}
Given a point $\xi\in\pt H$, denote by $\C_n(\xi)\subset\pt F$ the cylinder
set consisting of all the infinite words beginning with the initial segment
$\s_-(s_1)\s_0(s_1)\a(s_1,s_2)\dots\a(s_{n-1},s_n)\s_0(s_n)$ of
$\s^\infty(\xi)$ in the expansion \eqref{eq:asymp}. Then
$$
\O = \bigcap_n \bigcup_{\xi\in\pt H} \C_n(\xi) \;, \qquad  \De =
\ov{\bigcup_{\xi\in\pt H} \C_1(\xi) } \;.
$$
The cylinders $\C_n(\xi)$ are all open in $\pt F$, whence the claim.
\end{proof}

\begin{rem}
The Schreier limit set $\O$ is closed in $\pt F$ ($\equiv$ compact in the
relative topology) if and only if $\pt H$ is compact, i.e., if and only if $H$
is finitely generated (cf. \remref{rem:compct}).
\end{rem}

The identification $\pi$ of the group $F\cong\bA^*_r$ with the set
$\Paths_o(X)$ (\propref{pr:1}) obviously extends to an identification (also
denoted by $\pi$) of the boundary $\pt F\cong\bA^\infty_r$ with the set
$\Paths^\infty_o(X)$ of infinite paths without backtracking in $X$ issued from
$o$. In terms of this identification the sets $\O$ and $\De$ admit the
following descriptions:

\begin{prop} \label{pr:2}
The Schreier limit set $\O$ corresponds to the set of infinite paths without
backtracking in $X$ which pass infinitely often through the edges not in the
spanning tree $T$. The Schreier fundamental domain $\De$ corresponds to the
set of paths which always stay in $T$, i.e., which never pass through any of
the edges from $\Edges(X)\setminus \Edges(T)$.
\end{prop}

\begin{proof}
In view of the correspondence from \thmref{th:span}, formula \eqref{eq:asymp}
shows that if $\o=\s^\infty(\xi)\in\O$, then the associated path $\pi(\o)$
passes through the edges $\E_{s_1},\E_{s_2},\dots$ at the moments which
correspond to the letters $\s_0(s_1),\s_0(s_2),\dots$. Conversely, let us
record consecutively the edges $\E_{s_1},\E_{s_2},\dots$ through which the
path corresponding to $\o\in\pt F$ passes. Then $\o=\s^\infty(\xi)$ for
$\xi=s_1s_2\dots$.

In the same way one verifies the description of the set $\De$. A word
$\o\in\pt F$ begins with the segment $\s_-(s)\s_0(s)$ for a certain $s\in\bS$
if and only if the edge $\E_s$ is the first edge not in $T$ through which the
associated path passes.
\end{proof}

Following the above argument one also obtains a description of the translates
$h\De$ of the Schreier fundamental domain (cf. \lemref{lem:dec} and
\figref{fig:many}):

\begin{prop} \label{pr:trans}
For any $h=s_1s_2\dots s_n\in\bS^*_r\cong H$ the set $h\De$ corresponds to the
set of paths in $X$ which, starting from $o$, pass through the edges
$\E_{s_1},\E_{s_2},\dots,\E_{s_n}$ and follow edges in $T$ at all the other
times.
\end{prop}

Since the origin $o$ can be joined with any point $x\in X$ by a unique path in
the spanning tree $T$, the correspondence described in \propref{pr:1}
determines a natural embedding of $T$ into the Cayley graph $\G(F,\bA)$ such
that $o$ is mapped to the identity $e\in F$. Then the boundary ($\equiv$ the
space of ends) $\pt T$ becomes a subset of $\pt F$, and \propref{pr:2} implies

\begin{prop}
Under the above identification the Schreier fundamental domain $\De\subset\pt
F$ is homeomorphic to the boundary $\pt T$ of the spanning tree $T$.
\end{prop}

\propref{pr:2} and \propref{pr:trans} yield

\begin{thm} \label{th:1}
Given a spanning tree in the Schreier graph $X\cong H\bs F$, the associated
Schreier limit set $\O$ and the translates of the Schreier fundamental domain
$\De$ provide a disjoint decomposition of the boundary
\begin{equation} \label{eq:*}
\pt F = \left(\bigsqcup_{h\in H}h \De\right) \sqcup \O \;.
\end{equation}
\end{thm}

\subsection{Geodesic spanning trees and minimal Schreier systems}
\label{sec:minSch}

A spanning tree $T$ in a graph $X$ is called \emph{geodesic} (with respect to
a root vertex $o$), if $d_T(o,x)=d_X(o,x)$ for every vertex $x$ of $X$. A
geodesic spanning tree exists in any connected graph. For a Schreier graph
$X$, one possible way to construct a geodesic spanning tree is to use the fact
that its edges are labelled with letters from $\bA$. Then, taking for any
vertex $x\in X$ the lexicographically minimal among all the geodesic segments
joining the origin $o$ with $x$, the union of all such minimal segments is a
geodesic spanning tree in $X$.

In terms of the discussion from \secref{sec:span}, a spanning tree in the
Schreier graph $X$ is geodesic if and only if the corresponding Schreier
transversal is \emph{minimal}, i.e., the length of each representative is
minimal in its coset. The Schreier system of free generators $\S$ associated
with a minimal Schreier transversal (equivalently, with a geodesic spanning
tree in $X$) is called a \emph{minimal Schreier system}.

An important consequence of minimality is the inequality
\begin{equation} \label{eq:triangle}
\bigl| |\s_-(s)| - |\s_+(s)| \bigr| \le 1 \qquad \forall\,s\in\bS \;,
\end{equation}
which follows at once from the geometric interpretation given in
\secref{sec:Sch} (more precisely, from formula \eqref{eq:s-+}).

\medskip

We refer the reader to \cite[Section 3.2]{Magnus-Karrass-Solitar76} for a
definition and construction of \emph{Nielsen systems} of free generators in a
subgroup of a free group.

\begin{thm}[{\cite[Theorem 3.4]{Magnus-Karrass-Solitar76}}]\label{thm:lem1}
Any minimal Schreier system of generators in a subgroup $H$ of a free group
$F$ is a Nielsen system. Conversely, any Nielsen system of generators is (up
to a possible inversion of some elements) a minimal Schreier system.
\end{thm}

\begin{rem}
The interpretation of minimal Schreier systems in geometric terms of geodesic
spanning trees allows one to make the proof given in
\cite{Magnus-Karrass-Solitar76} (and reproducing the argument from
\cite{Karrass-Solitar58}) significantly simpler. For instance, the ``minimal
Schreier $\implies$ Nielsen'' part (another proof of which was first given in
\cite{Hall-Rado48}) becomes in these terms completely obvious.

\end{rem}

\emph{From now on, when considering a generating set $\S$ in a subgroup $H\le
F$, we shall always assume that $\S$ is a minimal Schreier system associated
with a geodesic spanning tree $T$ in the Schreier graph $X$.}

\section{Hopf decomposition of the boundary action}\label{sec:hopf}

The aim of this Section is to show that the decomposition of the boundary $\pt
F$ into a disjoint union of the Schreier limit set and the translates of the
Schreier fundamental domains obtained in \thmref{th:1} in fact provides the
Hopf decomposition of the boundary action of the subgroup $H$ with respect to
the uniform measure $\m$ (\thmref{th:2}).

\subsection{Conservativity and dissipativity}

Let $G$ be a countable group acting by \emph{measure class preserving
transformations} on a measure space $(\X,m)$, i.e., the measure $m$ is
\emph{quasi-invariant} under this action (for any group element $g\in G$ the
corresponding translated measure defined as $gm(A)=m(g^{-1}A)$ is equivalent
to $m$).

\medskip

\emph{Usually we shall denote the measure $m$ of a set $A$ just by $mA$,
although if necessary we may bracket either the measure or the set. Unless
otherwise specified, all the identities, properties etc.\ related to measure
spaces will be understood {\rm mod 0} (i.e., up to null sets)}.

\medskip

A measurable set $A\subset\X$ is called \emph{recurrent}  if for a.e.\ point
$x\in A$ the trajectory $Gx$ eventually returns to $A$, i.e., $gx\in A$ for a
certain element $g\in G$ other than the group identity $e$. Equivalently, $A$
is recurrent iff $A\subset\bigcup_{g\in G\setminus\{e\}} gA$. The opposite
notion is that of a \emph{wandering set}, i.e., a measurable set $A\subset\X$
with pairwise disjoint translates $gA,\,g\in G$.

The action of $G$ on $(\X,m)$ is called \emph{conservative} if any measurable
subset of positive measure is recurrent, and it is called \emph{dissipative}
if there exists a wandering set of positive measure. If the whole action space
is the union of translates of a certain wandering set, then the action is
called \emph{completely dissipative}.

\begin{rem} \label{rem:consup}
If a set $A\subset\X$ is recurrent with respect to a subgroup $G'\subset G$,
then it is obviously recurrent with respect to the whole group $G$. Therefore,
conservativity of the action of $G'$ implies conservativity of the action of
$G$.
\end{rem}

The action space always admits a unique \emph{Hopf decomposition}
$\X=\C\sqcup\D$ into a union of two disjoint $G$-invariant measurable sets
$\C$ and $\D$ (called the \emph{conservative} and the \emph{dissipative parts}
of the action, respectively) such that the restriction of the action to $\C$
is conservative and the restriction of the action to $\D$ is completely
dissipative, see \cite{Aaronson97} and the references therein. Hopf was also
the first to notice that for certain classes of dynamical systems the above
decomposition is trivial: any system from these classes is either conservative
or completely dissipative. In this situation one talks about the \emph{Hopf
alternative} (see \secref{sec:ergo} for more details).

\medskip

There is an important class of measure spaces called \emph{Lebesgue spaces}
(e.g., see \cite{Rohlin52,Cornfeld-Fomin-Sinai82}). Measure-theoretically
these are the measure spaces such that their non-atomic part is isomorphic to
an interval with the Lebesgue measure on it. There is also an intrinsic
definition of Lebesgue spaces based on their separability properties. However,
for our purposes it is enough to know that any \emph{Polish topological space}
(i.e., separable, metrizable, complete) endowed with a Borel measure is a
Lebesgue measure space, so that \emph{all the measure spaces considered in
this paper are Lebesgue.} A significant feature of Lebesgue spaces is that any
measure class preserving action of a countable group on a Lebesgue space
admits a (unique) \emph{ergodic decomposition} \cite[Theorem 6.6]{Schmidt77}.

For Lebesgue spaces the Hopf decomposition can also be described in terms of
the ergodic components of the action, see \cite{Kaimanovich10}. In the case
when the action is (essentially) \emph{free}, i.e., the stabilizers of almost
all points are trivial, this description is especially simple. Namely, $\C$ is
the union of all the ergodic components for which the corresponding
conditional measure is purely non-atomic, whereas $\D$ is the union of all the
purely atomic ergodic components (i.e., of the ergodic components which
consist of a single $G$-orbit; we shall call such orbits \emph{dissipative}).

Below we shall use the following explicit description of the conservative part
of an action in terms of its Radon--Nikodym derivatives.

\begin{thm}[\cite{Kaimanovich10}] \label{thp:3}
Let $(\X,m)$ be a Lebesgue space endowed with a free measure class preserving
action of a countable group $G$. Denote by $\mu_x,\,x\in\X$, the measure on
the orbit~$Gx$ defined as
$$
\mu_x(gx)=\frac{dg^{-1} m}{dm}(x)=\frac{dm(gx)}{dm(x)}
$$
(the measures $\mu_x$ corresponding to different points $x$ from the same
$G$-orbit are obviously proportional). Then for a.e.\ point $x\in X$ the
following conditions are equivalent:
\begin{itemize}
    \item[(i)]
    The orbit $Gx$ is dissipative;
    \item[(ii)]
    The measure $\mu_x$ is finite.
\end{itemize}
\end{thm}

\subsection{The uniform measure and the Busemann function}

Given a group element $g\in F$ we shall denote by $C_g=C_{\s(g)}\subset\pt F$
the associated \emph{cylinder set} of dimension $|g|$, which is the set of all
the infinite words which begin with the word $\s(g)$:
$$
C_g = \{\o\in\pt F: [\o]_{|g|}=g \} \;.
$$
Geometrically, $C_g$ is the ``shadow'' of $g$, i.e., the set of the endpoints
of all the geodesic rays issued from the group identity $e$ and passing
through the point~$g$.

Denote by $\m$ the probability measure on $\pt F\cong\bA^\infty_r$ which is
\emph{uniform with respect to the generating set $\A$}. In other words, all
the cylinder sets of the same dimension have equal measure
\begin{equation} \label{eq:cylmeas}
\m C_g = \frac1{2m(2m-1)^{|g|-1}} \qquad\forall\,g\in F\setminus\{e\}\;,
\end{equation}
where $m=|\A|$ is the number of generators of the group $F$.

\medskip

The Radon--Nikodym derivatives of $\m$ have a natural geometric
interpretation. Let us first remind the corresponding notions. Given a point
$\o\in\pt F$, the associated \emph{Busemann cocycle} is defined as
$$
\b_\o(g_1, g_2) = \lim_{g\to\o} \bigl[ d(g_2,g) - d(g_1,g) \bigr] = d(g_2,
g_1\wedge_\o g_2) - d(g_1, g_1\wedge_\o g_2) \;, \qquad g_1,g_2\in F \;,
$$
where $d(g,g')=|g^{-1}g'|$ is the distance in the Cayley graph $\G(F,\bA)$,
and $g_1\wedge_\o g_2$ denotes the starting point of the common part of the
geodesic rays $[g_1,\o)$ and $[g_2,\o)$, see \figref{fig:bus}. Thus,
$\b_\o(g_1, g_2)$ is a regularization of the formal expression $``d(g_2,\o) -
d(g_1,\o)\text{''}$. It can also be written as
$$
\b_\o(g_1,g_2) = b_\o(g_2) - b_\o(g_1) \;,
$$
where
$$
b_\o(g) = \b_\o(e,g) = \lim_{n\to\infty} \bigl[ d(g,[\o]_n) - n \bigr]
$$
is the \emph{Busemann function} associated with the point $\o\in\pt F$ (or,
geometrically, with the corresponding geodesic ray $[e,\o)$).

\begin{figure}[h]
\begin{center}
     \psfrag{e}[rb][rb]{$e$}
     \psfrag{a}[rb][rb]{$g_1$}
     \psfrag{b}[l][l]{$g_2$}
     \psfrag{g}[rt][rt]{$g_1\wedge_\o g_2$}
     \psfrag{o}[l][l]{$\o$}
          \includegraphics[scale=.75]{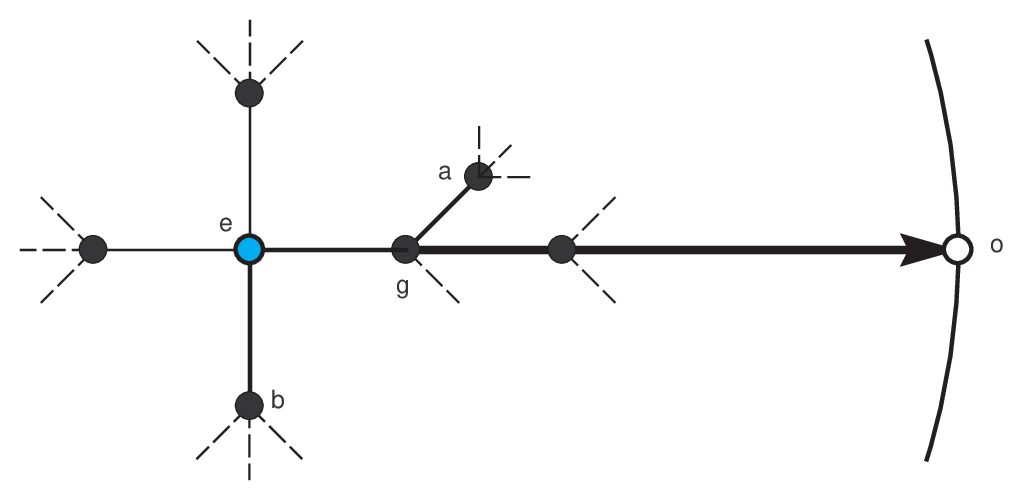}
          \end{center}
          \caption{}
          \label{fig:bus}
\end{figure}

For two words $w_1,w_2\in\bA^*\cup\bA^\infty$ denote by $w_1\wedge w_2$ their
\emph{confluent} (i.e., the longest common initial segment), and by
\begin{equation} \label{eq:groprod}
(w_1|w_2)=|w_1\wedge w_2|
\end{equation}
their \emph{Gromov product} \cite{Gromov87}, see \figref{fig:prod}. Then the
Busemann function is connected with the Gromov product by the formula
\begin{equation} \label{eq:GB}
b_\o(g) = |g| - 2 (g|\o) \;.
\end{equation}

\begin{figure}[h]
\begin{center}
     \psfrag{e}[rb][rb]{$e$}
     \psfrag{a}[r][r]{$g_2$}
     \psfrag{b}[r][r]{$g_1$}
     \psfrag{o}[rb][rb]{$g_1\wedge g_2$}
          \includegraphics[scale=.75]{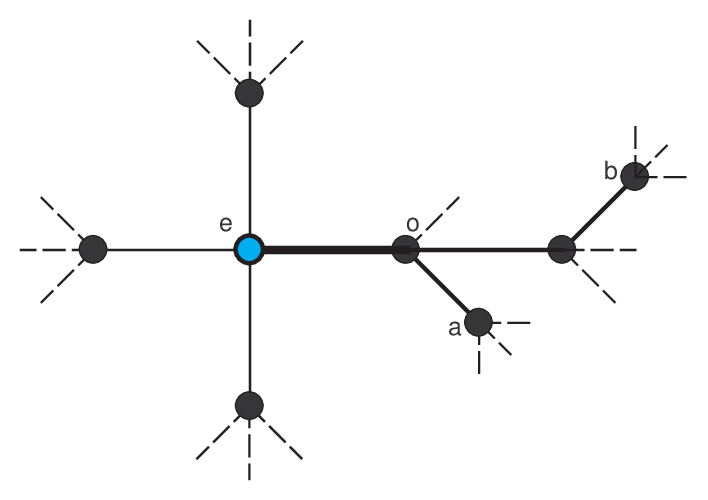}
          \end{center}
          \caption{}
          \label{fig:prod}
\end{figure}

The following property of the measure $\m$ is well-known, and can easily be
established by comparing measures of cylinder sets.

\begin{prop} \label{pr:RN}
The measure $\m$ is quasi-invariant with respect to the action of $F$ on $\pt
F$, and its Radon--Nikodym cocycle is
\begin{equation} \label{eq:RN=B}
\frac{dg\m}{d\m}(\o) = (2m-1)^{-b_\o(g)} \qquad \forall\, g\in F\;
\forall\,\o\in\pt F\;.
\end{equation}
\end{prop}

\begin{proof}
Since $gC_{g'}=C_{gg'}$ for any $g'\in F$ with $|g'|>|g|$, we have that under
this condition
$$
g\m \left(C_{g'}\right) = \m \left(g^{-1} C_{g'}\right) = \m C_{g^{-1}g'} =
\frac1{2m(2m-1)^{|g^{-1}g'|-1}} \;,
$$
whence
$$
\frac{g\m C_{g'}}{\m C_{g'}} = (2m-1)^{|g'|-|g^{-1}g'|} = (2m-1)^{-b_\o(g)}
\qquad \forall\,\o\in C_{g'} \;.
$$
\end{proof}

\begin{rem} \label{rem:RN}
The objects appearing in the left-hand and the right-hand sides of formula
\eqref{eq:RN=B} are of different nature. The Radon--Nikodym derivatives in the
left-hand side are \emph{a~priori} defined almost everywhere only, whereas the
Busemann function in the right-hand side is a \emph{bona fide} continuous
function on the boundary. It is a commonplace that such an equality is
interpreted as saying that there is a version of the Radon--Nikodym derivative
in the left-hand side given by the individually defined function in the
right-hand side.
\end{rem}

\subsection{Inequalities for the Busemann function}

The following auxiliary properties of the Busemann function are used in the
proof of \thmref{th:2} below and later on.

\begin{prop} \label{pr:RNOmega}
If $\o=\s^\infty(\xi)\in\O$ for $\xi=s_1s_2\dots\in\bS^\infty_r\cong\pt H$,
and $h=[\xi]_n=s_1s_2\dots s_n$, then $b_\o(h) \le 0$.
\end{prop}

\begin{proof}
Denote by $[x_i,y_i]=\E_{s_i}\in\Edges(X)\setminus\Edges(T)$ the oriented
edges corresponding to the generators $s_i\in\bS$ (see \thmref{th:span}). The
path $\pi(h)$ in $X$ (see \propref{pr:1}) starts from the origin $o$, passes
consecutively through the points $x_1,y_1,\dots,x_n,y_n$, and returns to $o$.
The segments $[o,x_1],[y_1,x_2],\dots,[y_{n-1},x_n],[y_n,o]$ in this path are
obtained by joining their endpoints in the spanning tree $T$, and the segments
$[x_1,y_1],\dots,[x_n,y_n]$ are just the corresponding edges $\E_{s_i}$, see
\lemref{lem:dec}. Denote by $D$ the length of the path $\pi(h)$ from the
beginning until the point $y_n$, and by $L$ the length of the remaining
segment $[y_n,o]$, so that the total length of this path is $|h|=D+L$. Since
$T$ is a \emph{geodesic} spanning tree (it is here that we use this
condition), $L$ is the distance between $y_n$ and $o$ in the graph $X$, so
that by the triangle inequality $L\le D$.

The infinite path $\pi(\o)$ also starts from the point $o$. It passes
consecutively through the points $x_1,y_1,\dots,x_n,y_n,\dots$. The segments
$[o,x_1],[y_1,x_2],\dots,[y_{n-1},x_n],\dots$ are obtained by joining their
endpoints in the spanning tree $T$, and the segments
$[x_1,y_1],\dots,[x_n,y_n],\dots$ are the edges $\E_{s_i}$. Therefore,
$(h|\o)\ge D$, and by \eqref{eq:GB}
$$
b_\o(h) = |h| - 2(h|\o) \le (D+L) -2D = L-D \le 0 \;.
$$
\end{proof}

\begin{prop} \label{pr:RNDelta}
If $\o\in\De$, then $b_\o(h)\ge 0$ for any $h\in H$.
\end{prop}

\begin{proof}
This is the same argument as in the proof of \propref{pr:RNOmega}. In view of
formula \eqref{eq:GB} we have to show that $|h|\ge 2(h|\o)$. Let us split the
cycle $\pi(h)$ in $X$ associated with $h$ into two parts. The first one is the
geodesic segment from $o$ to the beginning $x_1$ of the oriented edge
$[x_1,y_1]$ corresponding to the first letter $s_1\in\bS$ of $h$, and the
second one is the rest of $\pi(h)$. By the triangle inequality the length of
the second part is at least $d_X(o,x_1)$, so that the total length $|h|$ of
the path $\pi(h)$ is at least $2d_X(o,x_1)$. On the other hand, $(h|\o)\le
d_X(o,x_1)$, because $\o$ does not pass through $[x_1,y_1]$.
\end{proof}

\begin{prop} \label{pr:geod}
If a point $\o\in\pt F$ corresponds to a geodesic ray in $X$, then $b_\o(h)\ge
0$ for any $h\in H$.
\end{prop}

\begin{proof}
This is again the same argument consisting in comparing the length of a
geodesic subsegment in the cycle $\pi(h)$ with its total length. Let
$g=h\wedge\o$, and $x=og$. It means that the cycle $\pi(h)$ first follows the
path $\r$ determined by $\o$, until it reaches the point $x$, after which it
somehow returns to the origin $o$. Since $\r$ is a geodesic ray, the length of
the first part of the cycle (from the origin $o$ to the point $x$ along the
ray $\r$) does not exceed the length of the remaining part, whence
$(h|\o)=d_X(o,x)\le |h|/2$ which implies the claim.
\end{proof}

\subsection{The Hopf decomposition of the boundary action}

Now we are ready to prove the main result of this Section:

\begin{thm}\label{th:2}
The Schreier limit set $\O\subset\pt F$ determined by a geodesic spanning tree
in the Schreier graph $X\cong H\bs F$ ($\equiv$ by a minimal Schreier
generating system) coincides \textup{(mod 0)} with the conservative part of
the action of the subgroup $H\le F$ on the boundary $\pt F$ with respect to
the uniform measure $\m$.
\end{thm}

\begin{proof}
Since the Schreier fundamental domain $\De$ is measurable (\propref{pr:meas}),
\thmref{th:1} implies that the dissipative part $\D$ of the action is at least
the union $\bigcup_h h\De$, so that the conservative part $\C$ of the action
is contained in $\O$. For showing that $\C=\O$ we shall introduce the
$H$-invariant set
\begin{equation} \label{eq:Si}
\Si=\left\{\o\in\pt F: \sum_{h\in H} (2m-1)^{-b_\o(h)} = \infty \right\} \;.
\end{equation}
Any non-trivial element $g\in F$ has two fixed points on $\pt F$ (the
attracting and the repelling ones). Therefore, the boundary action is
essentially free with respect to any purely non-atomic quasi-invariant
measure. Thus, by \propref{pr:RN} and the criterion from \thmref{thp:3}, the
set $\Si$ coincides (mod 0) with the conservative part $\C$, and it remains to
show that $\O\subset\Si$, which follows at once from \propref{pr:RNOmega}
above.
\end{proof}

\begin{rem} \label{rem:1}
A priori, the set $\O$ depends on the choice of Nielsen--Schreier generating
system $\S$ in $H$ ($\equiv$ of a geodesic spanning tree in the Schreier graph
$\equiv$ of a minimal Schreier transversal). However, \thmref{th:2} shows that
the sets $\O$ for different choices of $\S$ only differ by a subset of
$\m$-measure $0$.
\end{rem}

\begin{rem}
It is likely that \thmref{th:2} also holds for many other boundary measures.
It would be interesting to investigate this question further. What (if any)
are the examples of purely non-atomic quasi-invariant measures on $\pt F$, for
which the conservative part is strictly smaller than the Schreier limit set
$\O$?
\end{rem}

\section{Limit sets and the core} \label{sec:limit-sets}

In this Section we compare the Schreier limit set $\O$ with several other
limit sets. In particular, we show that $\O$ (mod 0) coincides with the
horospheric limit sets (\thmref{th:coinc}). In this discussion we use
connections with random walks and with several extensions of the boundary
action.

\subsection{Hanging branches and the core} \label{sec:hang}

A \emph{branch} of a regular tree is a subtree which has one vertex (the
\emph{root}) of degree~$1$ and all the other vertices (which form its
\emph{interior}) are of full degree. Such a branch is uniquely determined by
its \emph{stem}, which is the oriented edge going from the root to the
interior of the branch, see \figref{fig:stem}.

\begin{figure}[h]
\begin{center}
          \includegraphics[scale=.75]{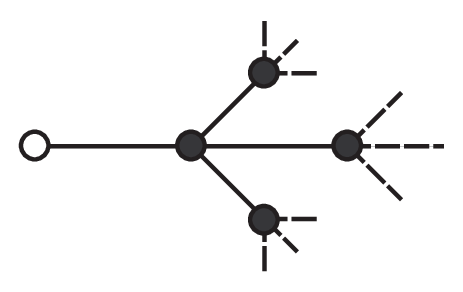}
          \end{center}
          \caption{}
          \label{fig:stem}
\end{figure}

\begin{defn}
A subgraph of the Schreier graph $X$ isomorphic to a branch in the Cayley
graph of $F$ (with its labelling) is called a \emph{hanging branch}. The
subgraph $X_*\subset X$ obtained by removing from $X$ all the hanging branches
(i.e., all their edges and all the interior vertices) is called the
\emph{core} of $X$.
\end{defn}

With the exception of the trivial case when the Schreier graph $X$ is a tree,
i.e., $H=\{e\}$ (which we shall always exclude below), the core $X_*$ is
non-empty. Any hanging branch is contained in a unique maximal hanging branch,
and maximal hanging branches are precisely those whose root belongs to the
core. In other words, the graph $X$ is obtained from the core $X_*$ by filling
the deficient valencies of the vertices of $X_*$ with maximal hanging branches
(so that all the degrees of the resulting graph have the full valency
$|\bA|$). Thus, since the Schreier graph $X$ is connected, its core $X_*$ is
also connected.

\begin{rem}
The definition of the core of a graph as what is left after removing all the
subtrees is due to Gersten \cite{Gersten83,Stallings83}. See
\cite{Stallings91,Kapovich-Myasnikov02} for an exposition of the ensuing
approach of Stallings to the study of subgroups of free groups based on the
notion of a \emph{folding} of graphs. Note that, following \cite[Section
7.2]{Stallings83} we are talking about the \emph{absolute} core of a graph
which is independent of the choice of a reference vertex. Some authors (e.g.,
\cite{Bahturin-Olshanskii10}) use a different definition, according to which
the (relative) core is the union of all reduced loops in the Schreier graph
$X$ starting from a chosen reference point $o\in X$. The absolute and the
relative cores coincide if and only if the reference vertex $o$ lies in the
absolute core.
\end{rem}

The following property is, of course, known to specialists (moreover, it is
basically the \emph{raison d'\^etre} of the definition of the core).

\begin{prop} \label{pr:finite}
A subgroup $H\le F$ is finitely generated if and only if the core $X_*$ of the
associated Schreier graph $X$ is finite.
\end{prop}

\begin{proof}
There is a natural one-to-one correspondence between spanning trees in the
Schreier graph $X$ and in the core $X_*$. Indeed, the restriction of any
spanning tree in $X$ to $X_*$ is a spanning tree in $X_*$. Conversely, any
spanning tree in $X_*$ uniquely extends to a spanning tree in $X$ by attaching
to it all the maximal hanging branches. Now, if the core is finite, then any
spanning tree in it ($\equiv$ the associated spanning tree in $X$) is obtained
by removing finitely many edges, so that the number of generators of $H$ is
finite (see \thmref{th:span}). Conversely, if $H$ is finitely generated, then
the core is contained in the finite union of the cycles in $X$ corresponding
to the generators of $H$.
\end{proof}

The definition of the core directly implies

\begin{lem} \label{lem:dec1}
All the paths $\f\in\Paths_o(X)\cup\Paths_o^\infty(X)$ (i.e., both finite and
infinite paths without backtracking issued from $o$) can be uniquely split
into the following three consecutive parts (some of which may be missing), see
\figref{fig:core}:
\begin{itemize}
\item[(i)]
the geodesic segment joining the origin $o$ with the root $o'\in X_*$ of the
maximal hanging branch which contains $o$ (if $o\notin X_*$ and $\f$ passes
through $X_*$);
\item[(ii)]
the part of $\f$ (possibly infinite) which is contained in $X_*$;
\item[(iii)]
the part of $\f$ (possibly infinite) entirely contained inside a certain
hanging branch.
\end{itemize}
\end{lem}

\begin{figure}[h]
\begin{center}
     \psfrag{o}[rb][rb]{$o$}
     \psfrag{O}[rb][rb]{$o'$}
     \psfrag{X}[l][l]{$X_*$}
          \includegraphics[scale=.6]{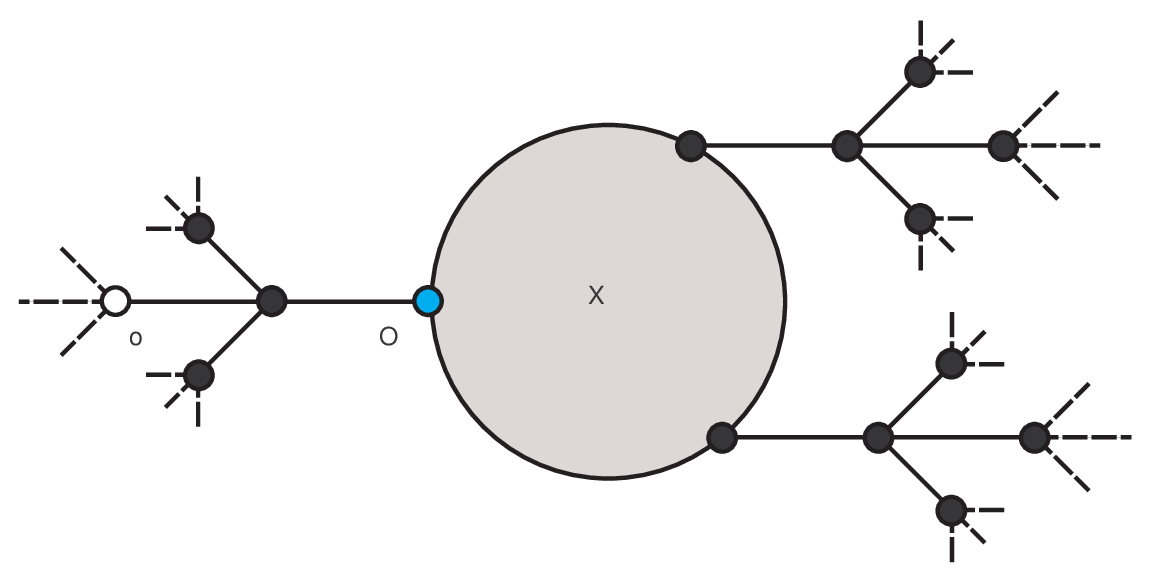}
          \end{center}
          \caption{}
          \label{fig:core}
\end{figure}

Below we shall also use the following

\begin{lem} \label{lem:smaller}
Let $\S$ be the system of generators of a subgroup $H\le F$ determined by a
spanning tree $T$ in the Schreier graph $X$. Then for any $s\in\S$ the
Schreier graph $X'$ of the group $H'=\la \S\setminus\{s\}\ra$ is obtained by
deleting from $X$ the edge $\E_s$ and attaching a hanging branch to each of
its endpoints.
\end{lem}

\begin{proof}
Being a Schreier graph, the edges of $X$ are labelled with letters from $\bA$.
This labelling extends to a labelling of $X'$ which makes of it the Schreier
graph of a certain subgroup of $F$. We can choose a spanning tree $T'$ in $X'$
by taking the union of $T$ and of the hanging branches added during the
construction of $X'$. The tree $T'$ determines then a set of generators $\S'$,
which, since the labellings of $X$ and $X'$ agree, coincides with
$\S\setminus\{s\}$.
\end{proof}

\subsection{The full limit set} \label{sec:limit-setsa}

\begin{defn} \label{def:full}
The \emph{limit set} $\La=\La_H\subset\pt F$ of a subgroup $H\leq F$ is the
set of all the limit points of $H$ with respect to the compactification $\wh
F=F\cup\pt F$ described in \secref{sec:bdrymap} (below we shall sometimes call
this limit set \emph{full} in order to distinguish it from other limit sets).
\end{defn}

The limit set $\La$ is closed and $H$-invariant. The following description is
actually valid for an arbitrary discrete group of isometries of a Gromov
hyperbolic space, see \cite{Gromov87,Bourdon95}:

\begin{thm} \label{th:minimal}
The action of $H$ on $\La$ is \emph{minimal} (there are no proper
$H$-invariant closed subsets), whereas the action of $H$ on the complement
$\pt F\setminus\La$ is \emph{properly discontinuous} (no orbit has
accumulation points).
\end{thm}

In our concrete situation the limit set $\La$ and a certain natural
fundamental domain in $\pt F\setminus\La$ admit the following very explicit
description (similar to \propref{pr:2}) in terms of the correspondence
$\o\mapsto\pi(\o)$ between $\pt F$ and the set $\Paths_o^\infty(X)$ of
infinite paths without backtracking in $X$ starting from the origin $o$ (see
\propref{pr:1} and the comment before \propref{pr:2}).

\begin{thm} \label{th:full}
\hfill
\begin{itemize}
  \item [(i)]
The full limit set $\La\subset\pt F$ corresponds to the set of paths from
$\Paths_o^\infty(X)$ which eventually stay inside the core $X_*$ (i.e., for
these paths the component (ii) from \lemref{lem:dec1} is infinite).
  \item [(ii)]
The full limit set $\La$ is the closure $\ov\O$ of the Schreier limit set.
  \item [(iii)]
The complement $\pt F\setminus\La$ is a disjoint union of $H$-translates of
the fundamental domain $\Th=\De\cap \left( \pt F \setminus\La \right)$. The
set $\Th$ is open and corresponds to the set of paths from $\Paths_o^\infty$
which do not pass through any of the edges from $\Edges(X)\setminus\Edges(T)$
and eventually stay inside a hanging branch (i.e., for which the component
(iii) from \lemref{lem:dec1} is infinite).
  \item [(iv)]
$\La=\O$ (equivalently, $\De=\Th$) if and only if $H$ is finitely generated.
\end{itemize}
\end{thm}

\begin{proof}
(i) Elements of $H$ correspond to cycles in $\Paths_o(X)$. By
\lemref{lem:dec1}, if $o\in X_*$ then any cycle from $\Paths_o(X)$ is entirely
contained in the core, and if $o\notin X_*$ then for any such cycle the
components (i) and (iii) described in \lemref{lem:dec1} are the geodesic
segments $[o,o']$ and $[o',o]$, respectively. Thus, the pointwise limit of any
sequence of such cycles (as their lengths go to infinity) is a path from
$\Paths_o^\infty(X)$ with infinite component~(ii). Conversely, if the $n$-th
point $\f(n)$ of a path $\f\in\Paths_o^\infty(X)$ belongs to $X_*$, then the
corresponding truncation $[\f]_n$ can be extended to a cycle without
backtracking (as otherwise $\f(n)$ must be inside a hanging branch), so that
$\f$ is a pointwise limit of cycles from $\Paths_o(X)$.

(ii) As it follows from \thmref{th:bdrymap}, $\O\subset\La$, it is non-empty
and $H$-invariant. Thus, in view of the minimality of $\La$
(\thmref{th:minimal}), $\ov\O=\La$.

(iii) Since $\Th\subset\De$, the fact that its $H$-translates are pairwise
disjoint and that their union is the complement of $\La$ follows at once from
\thmref{th:1}. The description of $\Th$ in terms of the associated subset of
$\Paths_o^\infty(X)$ is a combination of the description of the complement of
$\La$, which is (i) above, and of the description of the Schreier fundamental
domain $\De$ (\propref{pr:2}). Finally, $\Th$ is open because hanging branches
contain no edges from $\Edges(X)\setminus\Edges(T)$, so that if a path
$\f\in\Paths_o^\infty$ corresponds to a point $\o\in\Th$, then the whole open
cylinder $C_{[\o]_n}$ is also contained in $\Th$ for a sufficiently large~$n$.

(iv) We shall prove this claim in terms of the descriptions of the sets $\De$
and $\Th$ from \propref{pr:2} and from (iii) above, respectively. Therefore,
in view of \propref{pr:finite} we have to show that the core $X_*$ is finite
if and only if all the paths from $\Paths_o^\infty(X)$ confined to the
spanning tree $T$ eventually hit a certain hanging branch. Indeed, if $X_*$ is
finite, then the restriction of the spanning tree $T$ to $X_*$ is also finite,
so that any path as above must eventually leave $X_*$ and enter a hanging
branch. Conversely, if $X_*$ is infinite, then the restriction of the spanning
tree $T$ to $X_*$ is also infinite, so that there is an infinite path without
backtracking obtained by first joining $o$ with $X_*$ (if $o\notin X_*$) and
then staying inside the restriction of $T$ to $X_*$.
\end{proof}

\begin{rem}
Although the Schreier fundamental domain $\De$ is closed (\propref{pr:meas})
and contains the fundamental domain $\Th$, it is \emph{not} necessarily the
closure of $\Th$. For instance, it may happen that $\Th$ is empty (i.e.,
$\La=\pt F$), although $\De$ is not (and even has positive measure, see
\remref{rem:conv6}).
\end{rem}

\begin{cor} \label{cor:limit}
$\La=\pt F$ if and only if the Schreier graph $X$ has no hanging branches
(i.e., $X_*=X$).
\end{cor}

\begin{cor} \label{cor:6}
If $X$ has a hanging branch, then the boundary action of $H$ has a non-trivial
dissipative part.
\end{cor}

\begin{proof}
If $X$ has a hanging branch, then by \corref{cor:limit} $\La\neq\pt F$, i.e.,
the fundamental domain $\Th$ is non-empty. Since $\Th$ is open, and the
measure $\m$ has full support, $\m\Th>0$, so that $\Th$ is a non-trivial
wandering set.
\end{proof}

\begin{rem} \label{rem:conv6}
The converse of \corref{cor:6} is not true, see \exref{ex:mingrowthdiss}.
\end{rem}

\subsection{Radial limit set}

Specializing the type of convergence in the definition of the full limit set
(\defref{def:full}) one obtains subsets of $\La$ with various geometric
properties.

\begin{defn}
The \emph{radial limit set} $\La^{rad}\subset\pt F$ is the set of all the
accumulation points of the sequences of elements of $H$ which stay inside a
tubular neighbourhood of a certain geodesic ray in $F$.
\end{defn}

This definition in combination with \thmref{th:full}(i) implies

\begin{prop} \label{pr:LaRad}
The radial limit set $\La^{rad}\subset\pt F$ corresponds to the set of paths
$\f\in\Paths_o^\infty(X)$ which eventually stay inside the core $X_*$ and do
not go to infinity (i.e., $\liminf_n d_X(o,\f(n))<\infty$).
\end{prop}

\begin{prop} \label{pr:LaRad=La}
The radial limit set $\La^{rad}$ coincides with the full limit set $\La$ if
and only if the group $H$ is finitely generated.
\end{prop}

\begin{proof}
By \propref{pr:finite}, $H$ is finitely generated if and only if the core
$X_*$ is finite, which implies the claim in view of \propref{pr:LaRad}.
\end{proof}

\begin{rem} \label{rem:ss2}
According to a result of Beardon and Maskit \cite{Beardon-Maskit74}, a
Fuchsian group is finitely generated if and only if its limit set is the union
of the radial limit set and the set of parabolic fixed points. In our
situation there are no parabolic points, so that \propref{pr:LaRad=La} is a
complete analogue of this result. In other words, a subgroup of a finitely
generated free group is \emph{geometrically finite} (see \cite{Bowditch93}) if
and only if it is finitely generated. Note that the core can also be defined
as the quotient of the \emph{geodesic convex hull} of the full limit set $\La$
by the action of $H$, so that in our situation geometrical finiteness of $H$
coincides with its \emph{convex cocompactness} (i.e., finiteness of the core).
\end{rem}

\subsection{Horospheric limit sets} \label{sec:horolim}

\begin{defn}
The \emph{horosphere} passing through a point $g\in F$ and centered
at a point $\o\in\pt F$ is the corresponding level set of the
Busemann cocycle $\b_\o$:
$$
\Hor_\o(g) = \{g'\in F: \b_\o(g,g')=0 \} \;.
$$
In the same way, the \emph{horoballs} in $F$ are defined as
$$
\HBall_\o(g) = \{g'\in F: \b_\o(g,g')\le 0 \} \;.
$$
\end{defn}

Restricting converging sequences to horoballs in $F$ provides us with
\emph{horo\-sphe\-ric limit points}. Unfortunately, the situation here is more
complicated than with the radial limit points, and we have to define two
different horospheric limit sets.

\begin{defn}
The \emph{small} (resp., \emph{big}) \emph{horospheric limit set}
$\La^{horS}=\La^{horS}_H$ (resp., $\La^{horB}=\La^{horB}_H$) of a subgroup
$H\leq F$ is the set of all the points $\o\in\pt F$ such that any (resp., a
certain) horoball centered at $\o$ contains infinitely many points from $H$.
\end{defn}

In terms of the Busemann function a point $\o\in\pt F$ belongs to
$\La^{horS}$ (resp., to $\La^{horB}$) if for any (resp., a certain)
$N\in\Z$ there are infinitely many points $h\in H$ with $b_\o(h)\le
N$.

\begin{rem} \label{rem:ss1}
Usually our small horospheric limit set is called just the
\emph{ho\-ro\-sphe\-ric limit set}, and in the context of Fuchsian and
Kleinian groups its definition, along with the definition of the radial limit
set, goes back to Hedlund \cite{Hedlund36}. Following \cite{Matsuzaki02} we
call it \emph{small} in order to better distinguish it from the \emph{big}
one, which, although apparently first explicitly introduced by Tukia
\cite{Tukia97}, essentially appears already in Pommerenke's paper
\cite{Pommerenke76}. See \cite{Starkov95,Dalbo-Starkov00} for a detailed
discussion of various kinds of limit points for Fuchsian groups.
\end{rem}

The horospheric limit sets $\La^{horS},\La^{horB}$ are obviously $H$-invariant
and contained in the full limit set $\La$ (since the only boundary
accumulation point of any horoball is just its center). The following theorem
describes the relationship between the full limit set $\La$, the radial limit
set $\La^{rad}$, the both horospheric limit sets, the Schreier limit set $\O$
and the set~$\Si$ \eqref{eq:Si}.

\begin{thm} \label{th:3}
One has the inclusions
$$
\La^{rad}\subset\La^{horS} \subset \O \subset \La^{horB} \subset \Si
\subset \La \;.
$$
\end{thm}

\begin{proof} \hfill

\medskip

$\La^{rad}\subset\La^{horS}$. Obvious.

\medskip

$\La^{horS} \subset \O$. It follows from \thmref{th:1} and
\propref{pr:RNDelta}.

\medskip

$\O \subset \La^{horB}$. This inclusion was actually already
established in the course of the proof of \thmref{th:2}.

\medskip

$\La^{horB} \subset \Si$. Obvious.

\medskip

$\Si \subset \La$. Clearly, we may assume that $\La\neq\pt F$. If
$\o\notin\La$, then $\max_{h\in H} (h|\o)<\infty$ (for, if
$(h_n|\o)\to\infty$, then $h_n\to\o$). Thus, by formula \eqref{eq:GB} for any
$\o\notin\La$ convergence of the series \eqref{eq:Si} from the definition of
the set $\Si$ is equivalent to convergence of the \emph{Poincar\'e series}
$\sum_{h\in H} (2m-1)^{-|h|}$. Now, if $\La\neq\pt F$, then by
\corref{cor:limit} and \corref{cor:6} the boundary action has a non-trivial
dissipative part, and the Poincar\'e series is convergent by
\corref{cor:Pconv} below.
\end{proof}

We shall show in \secref{sec:exincl} later on that all the inclusions in
\thmref{th:3} are, generally speaking, strict. Nonetheless,

\begin{thm} \label{th:coinc}
The sets $\La^{horS},\O,\La^{horB},\Si$ all coincide
$\m$-\textup{mod 0}.
\end{thm}

\begin{proof}
As it follows from \thmref{thp:3}, \propref{pr:RN} and \thmref{th:2}, the sets
$\Si$ and~$\O$ coincide $\m$-mod 0. We shall show that
$\m(\La^{horB}\setminus\La^{horS})=0$ which would imply the claim. Indeed, on
the $H$-invariant set $A=\La^{horB}\setminus\La^{horS}$, which is contained
(mod~0) in the conservative part of the action, the projection $\wt{\pt
F}\cong\pt F\times\Z\to\pt F$ admits a measurable $H$-equivariant section
(which consists in assigning to a boundary point $\o$ the ``smallest''
horosphere centered at $\o$ and containing an infinite number of points from
$H$). It implies that the ergodic components of the skew action of $H$ on
$A\times\Z\subset\wt{\pt F}$ are given by taking this section and its shifts
over the ergodic components of the action of $H$ on $A$, the latter being
impossible by \thmref{th:horocycle} on a set of positive measure.
\end{proof}

\begin{rem} \label{rem:s1}
Coincidence (mod 0) of the conservative part of the boundary action with the
big horospheric limit set $\La^{horB}$ is actually true in much greater
generality of an arbitrary Gromov hyperbolic space endowed with a
quasi-conformal boundary measure \cite{Kaimanovich10}. The proof uses the fact
that, by definition, the logarithms of the Radon--Nikodym derivatives of this
measure are (almost) proportional to the Busemann cocycle, in combination with
the criterion from \thmref{thp:3}.
\end{rem}

\subsection{Boundary action and random walks} \label{sec:RW}

An important aspect of the boundary behaviour is related to the asymptotic
properties of \emph{random walks} on the group $F$ and on the Schreier graph
$X$ (see \cite{Kaimanovich-Vershik83,Kaimanovich00a} and the references
therein for the general background), and, in particular, to the fact that the
uniform measure on the boundary $\m$ can be interpreted as the harmonic
measure of the simple random walk on $F$.

Let $\mu$ be the probability measure on the group $F$ equidistributed on the
generating set~$\bA$. Then the random walk $(F,\mu)$ is precisely the
\emph{simple random walk} on the Cayley graph $\G(F,\bA)$, i.e., for any $g\in
F$ the transition probability $\pi_g=g\mu$ is equidistributed on the set of
neighbors of $g$ in the graph $\G(F,\bA)$. Moreover, at each point $g\in F$
the increment $\h\in\bA$ is precisely the label of the edge along which the
random walk moves to a new position.

The following result (which we reformulate in modern terms) is due to Dynkin
and Maljutov \cite{Dynkin-Malutov61}:

\begin{thm} \label{th:DM}
Sample paths of the simple random walk $(F,\mu)$ converge a.e.\ to the
boundary $\pt F$, the hitting distribution is the uniform measure $\m$, and
the space $(\pt F,\m)$ is isomorphic to the Poisson boundary of this random
walk.
\end{thm}

There is an obvious one-to-one correspondence between the harmonic functions
of the simple random walk on the Schreier graph $X$ (the definition of which
--- the same as for the simple random walk $(F,\mu)$ --- takes into account
eventual loops and multiple edges in~$X$) and $H$-invariant $\mu$-harmonic
functions on $F$. This situation is a very specific case of a general theory
of covering Markov operators developed in \cite{Kaimanovich95}. In particular,
\cite[Theorem 2.1.4]{Kaimanovich95} implies:

\begin{thm} \label{th:pois}
The space of ergodic components of the action of $H$ on $(\pt F,\m)$ is
canonically isomorphic to the Poisson boundary of the simple random walk on
the Schreier graph~$X$. In particular, the action is ergodic if and only if
this random walk is Liouville ($\equiv$ has no non-constant bounded harmonic
functions).
\end{thm}

Yet another corollary of the general theory (see \cite[Theorem
3.3.3]{Kaimanovich95}) is

\begin{thm} \label{th:normalcons}
The action of any non-trivial normal subgroup $H\lhd F$ on $(\pt F,\m)$ is
conservative.
\end{thm}

\begin{rem} \label{rem:s3}
A similar result in the hyperbolic setup is due to Matsuzaki: the boundary
action of any normal subgroup of a divergent type discrete group $G$ of
isometries of the hyperbolic space $\HH^{d+1}$ is conservative with respect to
the associated Patterson measure class. It was first established when the
critical exponent $\d$ of the group $G$ is $d$ (i.e., the corresponding
Patterson measure class coincides with the boundary Lebesgue measure class)
\cite{Matsuzaki93}, and later extended to the case of an arbitrary $\d$
\cite{Matsuzaki02}. For $\d\ge d/2$ this result also readily follows from the
theory of covering Markov operators \cite[Theorem~4.2.4]{Kaimanovich95}.
\end{rem}

The situation when $\m\La=0$ can be completely characterized in terms of the
simple random walk on the Schreier graph.

\begin{prop} \label{pr:RWL}
$\m\La=0$ if and only if a.e.\ sample path of the simple random walk on~$X$
eventually stays inside a certain hanging branch.
\end{prop}

\begin{proof}
Using the labelling of edges, any sample path $(x_n)$ of the simple random
walk on $X$ can be uniquely lifted to a sample path $(g_n)$ of the simple
random walk on $F$. The latter a.e.\ converges to a boundary point $\o\in\pt
F$, and the distribution of $\o$ is precisely the measure $\m$
(\thmref{th:DM}). A sample path $(x_n)$ eventually stays inside a hanging
branch if and only if the path $\pi(\o)$ does so, whence the claim in view of
\thmref{th:full}(i).
\end{proof}

\begin{cor} \label{cor:RWD}
If the simple random walk on the Schreier graph $X$ is such that a.e.\ sample
path eventually stays inside a certain hanging branch, then the boundary
action is completely dissipative.
\end{cor}

\begin{rem}
The converse of \corref{cor:RWD} is not true, see \exref{ex:mingrowthdiss}.
\end{rem}

\corref{cor:6} implies that the boundary action of any finitely generated
group $H$ of infinite index has a non-trivial dissipative part. Indeed, since
$H$ is finitely generated, the core $X_*$ of the Schreier graph $X$ is finite
(\propref{pr:finite}). As $H$ is of infinite index, the graph $X$ is infinite
and so necessarily has a hanging branch.

In fact, the following dichotomy completely describes the conservativity
properties of the boundary action for finitely generated groups:

\begin{thm} \label{th:pr7}
If $H$ is finitely generated, then the Hopf alternative holds: either
\begin{itemize}
\item[(i)] $H$ is of finite index and its boundary action is ergodic (therefore,
conservative),
\end{itemize}
or
\begin{itemize}
\item[(ii)] $H$ is of infinite index and its boundary action is
completely dissipative.
\end{itemize}
\end{thm}

\begin{proof}
If a subgroup $H$ has a finite index, then the associated Schreier graph $X$
is finite, and therefore the simple random walk on it is Liouville, so that
the boundary action of~$H$ is ergodic by \thmref{th:pois}.

If $H$ is finitely generated of infinite index, then $X$ in an infinite graph
consisting  of a finite core $X_*$ and some hanging branches glued to it (see
\propref{pr:finite}). Since the simple random walk in any hanging branch is
transient, the simple random walk on $X$ is also transient, which implies that
a.e.\ sample path eventually stays inside a certain hanging branch, whence the
claim by \corref{cor:RWD}.
\end{proof}

\begin{rem}
In view of \propref{pr:RWL} the dichotomy from \thmref{th:pr7} can be
reformulated in the following way: for a finitely generated group $H$ either
the Schreier graph~$X$ is finite, or else $\m\La=0$.
\end{rem}

\begin{rem}
For infinitely generated subgroups this dichotomy does not hold, for instance,
see \exref{ex:consdiss}.
\end{rem}

\begin{rem} \label{rem:KarSol}
An unexpected application of \thmref{th:pr7} is a one line conceptual proof of
an old theorem of Karrass and Solitar \cite{Karras-Solitar57}: if $H\le F$ is
finitely generated and contains a non-trivial normal subgroup, then $H$ is of
finite index in $F$ (if $H$ itself is a normal subgroup, this was first proved
by Schreier in his famous 1927 paper \cite{Schreier27}). Indeed, if~$H$
contains a normal subgroup, then its boundary action is conservative by
\thmref{th:normalcons} and \remref{rem:consup}. Since $H$ is finitely
generated, by \thmref{th:pr7} it must be of finite index in $F$. Note that a
recent far-reaching generalization of the Karrass--Solitar theorem
\cite[Corollary 5.13]{Peterson-Thom07} in particular extends it to all
subgroups $H\le F$ with the property that the set $gH\cap Hg$ is infinite for
all $g\in F$. It would be interesting to compare this property with the
conservativity of the boundary action.
\end{rem}

\begin{rem}
Any infinitely generated subgroup $H$ of infinite index with conservative
boundary action readily provides the following example: the action of any
finitely generated subgroup of $H$ is completely dissipative by
\thmref{th:pr7} in spite of conservativity of the action of the whole group
$H$.
\end{rem}

\subsection{Extensions of the boundary action}\label{sec:ergo}

There are two extensions of the boundary action which have natural geometric
interpretations. The first one is the action of $H$ on the space $\pt^2
F:=(\pt F\times\pt F)\setminus\diag$, or, in other words, on the space of
bi-infinite geodesics in the Cayley graph $\G(F,\bA)$. We shall endow it with
the square $\m^2$ of the measure~$\m$.

Ergodicity of this action is equivalent to ergodicity of the (discrete)
geodesic flow on~$X$. The study of the ergodic properties of the action of a
discrete group of hyperbolic isometries of $\mathbb H^n$ on $\pt ^2 \mathbb
H^n$ has a long history beginning with the pioneering works of Hedlund and
E.~Hopf in the 30's for Fuchsian groups. Its current state is given by the
so-called \emph{Hopf--Tsuji--Sullivan theorem}, e.g., see \cite{Sullivan81,
Nicholls89, Kaimanovich94} and the references therein. Analogous results for
the action of a subgroup $H$ of a free group on $\pt ^2 F$ with respect to the
measure $\m^2$ were obtained by Coornaert and Papadopoulos
\cite[Corollaire~D]{Coornaert-Papadopoulos96}. The results from
\cite{Coornaert-Papadopoulos96} actually follow from more general
considerations of Kaimanovich \cite[Theorem~3.3 and the discussion in
Section~3.3.3]{Kaimanovich94}, where the general case of the harmonic measure
of a covering Markov operator on a Gromov hyperbolic space was treated. In our
situation these results can be summarized as the following analogue of the
Hopf--Tsuji--Sullivan theorem:

\begin{thm}[\cite{Kaimanovich94,Coornaert-Papadopoulos96}] \label{th:HTS}
The action of $H$ on $(\pt^2 F,\m^2)$ is either ergodic (therefore,
conservative) or completely dissipative (the \emph{Hopf alternative}). If the
action is ergodic, then
\begin{itemize}
    \item[(i)]
    $\m\La^{rad}=1$;
    \item[(ii)]
    The simple random walk on the Schreier graph $X$ is recurrent;
    \item[(iii)]
    The \emph{Poincar\'e series} $\sum_{h\in H} (2m-1)^{-|h|}$ diverges.
\end{itemize}
Alternatively, if the action is completely dissipative, then
\begin{itemize}
    \item[(i$'$)]
    $\m\La^{rad}=0$;
    \item[(ii$'$)]
    The simple random walk on the Schreier graph $X$ is transient;
    \item[(iii$'$)]
    The Poincar\'e series converges.
\end{itemize}
\end{thm}

\begin{cor} \label{cor:Pconv}
If the action of $H$ on $(\pt F,\m)$ is dissipative, then the Poincar\'e
series converges.
\end{cor}

\begin{proof}
Since the action of $H$ on $(\pt^2 F,\m^2)$ projets to the action on $(\pt
F,\m)$, dissipativity of the latter implies dissipativity of the former.
\end{proof}

\begin{rem}
The action of $F$ on $\pt^3 F$ (and therefore on all the higher products) is
properly discontinuous in view of the existence of an equivariant
\emph{barycenter map} $\pt^3 F\to F$. Hence, it is dissipative for any purely
non-atomic measure.
\end{rem}

Yet another extension of the boundary action is obtained by taking the skew
action of~$H$ on $\wt{\pt F}:=\pt F\times\Z$ determined by the Busemann
cocycle. Geometrically this action is just the action of $H$ on the space of
horospheres in $F$ (see \secref{sec:horolim}). The space~$\wt{\pt F}$ is
endowed with the natural measure $\wt\m$ which is the product of $\m$ and the
counting measure on $\Z$. The ergodic properties of this action essentially
coincide with the ergodic properties of the boundary action, namely:

\begin{thm} \label{th:horocycle}
Let $\pt F=\C\cup\D$ be the decomposition of the boundary $\pt F$ into the
conservative and the dissipative parts of the $H$-action. Then the
conservative and the dissipative parts of the action of $H$ on the space
$\wt{\pt F}$ are $\C\times\Z$ and $\D\times\Z$, respectively. The conservative
ergodic components of the $H$-action on $\wt{\pt F}$ are the preimages of the
ergodic components of the $H$-action on $\C$ under the projection $\wt{\pt
F}\to\pt F$. In particular, the action of $H$ on $\C\times\Z$ is ergodic if
and only if the action of $H$ on $\C$ is ergodic.
\end{thm}

The proof of this theorem almost \emph{verbatim} coincides with the proof of
the analogous result for the boundary actions of Fuchsian groups
\cite[Theorem~4.2]{Kaimanovich00} based in turn on Sullivan's proof in
\cite{Sullivan82} (see also \cite{Sullivan81}) of the fact that the boundary
action of a Kleinian group is of type $\textup{III}_1$ on its conservative
part with respect to the Lebesgue measure. The only difference is that in our
situation the range of the Busemann cocycle is $\Z$ (rather than $\Re$ as in
the case of manifolds), so that, in particular, the type of the boundary
action of $H$ on the conservative part is $\textup{III}_\l$ with
$\l=\log(2m-1)$.

\subsection{Examples} \label{sec:exincl}

We shall now give examples showing that all the sets from \thmref{th:3} are,
in general, pairwise distinct.

\begin{ex}
$\La^{rad}\neq\La^{horS}$. By \thmref{th:2}, \thmref{th:coinc},
\thmref{th:pois} and \thmref{th:HTS}, if the simple random walk on $X$ is
transient, but still has the Liouville property, then $\m\La^{rad}=0$, whereas
$\m\La^{horS}=1$. Such examples are readily available already in the case when
the subgroup $H$ is normal and $X$ is the quotient group, see
\cite{Kaimanovich-Vershik83}.
\end{ex}

\begin{ex}
$\La^{horS} \neq \O$. Let us take two geodesic rays $\r_1,\r_2\cong\Z_+$
joined at the origin~$o$, so that $\r_1(0)=\r_2(0)=o$. We construct the
Schreier graph $X$ by first adding the edges $[\r_1(n+1),\r_2(n)]$ for all
$n>0$ and then filling all the deficient valencies with hanging branches. The
geodesic spanning tree $T$ is obtained from $X$ by removing all the edges
$[\r_1(n),\r_1(n+1)]$ for $n>0$, see \figref{fig:ex17}. Let $\o_1\in\pt F$ be
the boundary point which corresponds to the ray $\r_1$ considered as a path
without backtracking in $X$. Then $\o_1\in\O$ because $\r_1$ passes through an
infinite number of edges from $\Edges(X)\setminus\Edges(T)$. On the other
hand, since $\o_1$ corresponds to a geodesic in $X$, $\o\notin\La^{horS}$ by
\propref{pr:geod}.
\end{ex}

\begin{figure}[h]
\begin{center}
     \psfrag{o}[tr][tr]{$o$}
     \psfrag{r1}[Bl][Bl]{$\rho_1(n+1)$}
     \psfrag{r2}[tl][tl]{$\rho_2(n)$}
     \psfrag{o1}[tl][l]{$\o_1$}
     \psfrag{o2}[tl][l]{$\o_2$}
          \includegraphics[scale=.75]{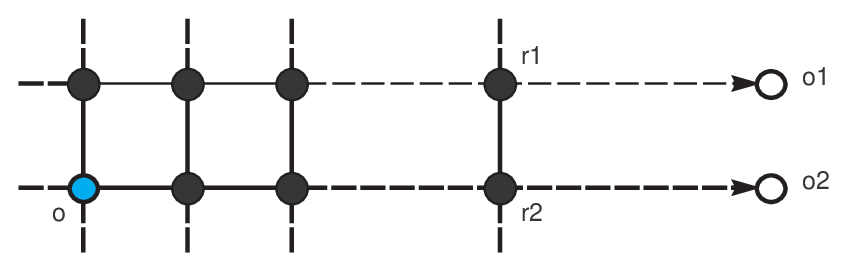}
          \end{center}
          \caption{}
          \label{fig:ex17}
\end{figure}

\begin{ex}
$\O \neq \La^{horB}$. Let $X$ be the same graph as in the previous example.
Take the boundary point $\o_2\in\pt F$ which corresponds this time to the ray
$\r_2$, see \figref{fig:ex17}. Then $\o_2\notin\O$. On the other hand, for any
$n>0$ the generator $s_n$ corresponding to the edge $[\r_1(n),\r_1(n+1)]$ has
the property that $b_{\o_2}(s_n)=2$, whence $\o_2\in\La^{horB}$.
\end{ex}

\begin{ex}
$\La^{horB} \neq \Si$. Let us take a geodesic ray $\r\cong\Z_+$ starting at
the origin $o=\r(0)$ and an integer sequence $d_n$ (to be specified later). We
construct the Schreier graph $X$ by first attaching to each vertex
$\r(n),\,n>0,$ a loop of length $2d_n+1$ and then filling all the deficient
valencies with hanging branches. The spanning geodesic tree $T$ is obtained by
removing from each such loop the middle edge $\E_n$, so that in $T$ there are
two segments of length $d_n$ attached to each point $\r(n)$, see
\figref{fig:ex19}. Denote by $s_n$ the corresponding generators. Let $\o\in\pt
F$ be the boundary point corresponding to the ray~$\r$. Then for any $h\in
H\cong\bS^*_r$ one has the inequality $b_\o(h)\ge\sum_n t_n(2d_n+1)$, where
$t_n$ is the number of occurrences of $s_n^{\pm 1}$ in $h$. Indeed, by
\eqref{eq:GB} $b_\o(h)=|h|-2(h|\o)$. We can write $|h|=l+\sum_n t_n(2d_n+1)$,
where $l$ is the sum of the lengths of the pieces of the associated path in
$X$ which correspond to moving along the ray $\r$. The latter sum contains the
term $(h|\o)$ which corresponds to the confluent of $h$ and $\r$, and, since
$h$ is a cycle, one has $l\ge 2(h|\o)$, which implies the desired inequality.
Now, if $d_n\uparrow\infty$, then $\o\notin\La^{horB}$. On the other hand,
$b_\o(s_n)=2d_n+1$, and one can still choose $d_n$ is such a way that $\sum_n
(2m-1)^{-2d_n} = \infty$, so that $\o\in\Si$.
\end{ex}

\begin{figure}[h]
\begin{center}
     \psfrag{o}[tr][tr]{$o$}
     \psfrag{r1}[t][t]{$\rho(1)$}
     \psfrag{r2}[t][t]{$\rho(n)$}
     \psfrag{E1}[b][b]{$\E_1$}
     \psfrag{E2}[b][b]{$\E_n$}
     \psfrag{o2}[tl][l]{$\o$}
          \includegraphics[scale=.75]{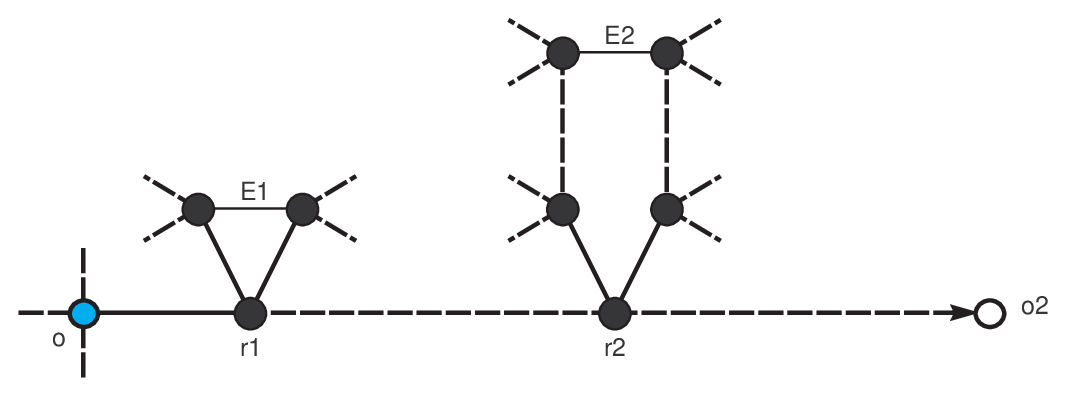}
          \end{center}
          \caption{}
          \label{fig:ex19}
\end{figure}

\begin{ex}
$\Si \neq \La$. Since $\Si$ coincides (mod 0) with the conservative
part of the action (see \thmref{th:coinc}), whenever the boundary
action is completely dissipative we have $\m\Si=0$. On the other
hand, in this situation it is still possible that $\La=\pt F$, i.e.,
that the Schreier graph has no hanging branches
(\corref{cor:limit}), see \exref{ex:mingrowthdiss}.
\end{ex}

\section{Geometry of the Schreier graph} \label{sec:Sch-graph-geometry}

In this Section we shall study the relationship between the ergodic
properties of the action of the subgroup $H$ on the boundary $(\pt
F,\m)$ and the geometry of the Schreier graph $X=H\bs F$ (in
particular, its quantitative characteristics).

\subsection{Growth and cogrowth} \label{sec:grcogr}

Denote by $S_F^n$ and $B_F^n$ (resp., $S_X^n$ and $B_X^n$) the
sphere and the ball of radius $n$ in the Cayley graph $\G(F,\bA)$
(resp., in the Schreier graph $X=\G(H\bs F,\bA)$) centered at the
group identity $e$ (resp., at the point $o=H$). The
\emph{exponential volume growth rate} of $X$ is
$$
v_X=\limsup\left|B_X^n\right|^{1/n} \le v_F=2m-1 \;.
$$
The \emph{cogrowth rate} of $X$ ($\equiv$ the growth rate of $H$ in
$F$) is
$$
v_H = \limsup_{n\to\infty} \left|H\cap B_F^n\right|^{1/n} \le 2m-1
\;,
$$
and it is the inverse of the radius of convergence of the
\emph{cogrowth series}
$$
\mathcal G_H(z)=\sum_n |H\cap S_F^n|z^n \;.
$$

Denote by $\r_F$ (resp., $\r_X$) the spectral radius of the simple random walk
on $F$ (resp., on $X=\G(H\bs F,\bA)$), so that $\sqrt{2m-1}/m=\r_F\le\r_X\le
1$ \cite{Kesten59}. By \cite{Grigorchuk78,Grigorchuk80a}
\begin{equation} \label{eq:rhov}
\r_X = \left\{%
\begin{array}{ll}
    \frac{\sqrt{2m-1}}{m} &, \quad 1\le v_H\le \sqrt{2m-1} \;,
    \vphantom{\text{\Huge (}}\\
    \frac{\sqrt{2m-1}}{2m}\left( \frac{\sqrt{2m-1}}{v_H}
    + \frac{v_H}{\sqrt{2m-1}} \right)
    &, \quad \sqrt{2m-1}\le v_H\le 2m-1 \;, \vphantom{\text{\Huge (}} \\
\end{array}%
\right.
\end{equation}
which implies that $\r_X=\r_F$ if and only if $v_H\le\sqrt{2m-1}$, and
$\r_X=1$ (i.e., the graph~$X$ is \emph{amenable}) if and only if $v_H=2m-1$.
Recall that an infinite connected $(2m)$-regular graph $X$ is called
\emph{Ramanujan} if $\r_X=\sqrt{2m-1}/m$. Therefore, a Schreier graph $X=H\bs
F$ is Ramanujan if and only if the growth of the corresponding subgroup $H\le
F$ satisfies the inequality $v_H\le \sqrt{2m-1}$.

\subsection{Dissipativity}

\begin{thm} \label{th:4}
If $v_H<\sqrt{2m-1}$ then the boundary action of $H$ is completely
dissipative.
\end{thm}

\begin{proof}
For $h=s_1s_2\dots s_n\in\bS_r^*\cong H$ let $C_h\subset\pt H$ be
the corresponding cylinder set, so that
$$
\O = \bigcup_{h\in H} \s^\infty(C_h) \;.
$$
By \eqref{eq:asymp}
$$
\s^\infty(C_h) \subset C_g \;,
$$
where $C_g\subset\pt F$ is the cylinder set based at the word
$$
g = \s_-(s_1)\s_0(s_1)\a(s_1,s_2)\s_0(s_2) \dots \a(s_{n-1},s_n)\s_0(s_n) \in
\bA^*_r\cong F \;.
$$
Since $|g|\ge |\s(h)|/2$ (see the first half of the proof of
\propref{pr:RNOmega}),
$$
\m\s^\infty(C_h) \le \m C_g = \frac{1}{2m(2m-1)^{|g|-1}} \leq
\frac{1}{(2m-1)^{|\s(h)|/2}} \;,
$$
whence
$$
\sum_{h\in H} \m\s^\infty(C_h) < \infty \;,
$$
which implies the claim by Borel--Cantelli lemma, because any point
from $\pt H$ belongs to infinitely many cylinders $C_h$.
\end{proof}

\begin{rem}
The converse of \thmref{th:4} is not true, see \exref{ex:amendiss}.
\end{rem}

\begin{rem} \label{rem:q}
The upper bound $\sqrt{2m-1}$ in \thmref{th:4} is optimal. Indeed, if $N\lhd
F$ is a non-trivial normal subgroup, then its boundary action is always
conservative, see \thmref{th:normalcons}. On the other hand, $N$ being
non-amenable, $\r_{F/N}>\r_F$ by Kesten's theorem \cite{Kesten59}, whence by
formula \eqref{eq:rhov} $v_N>\sqrt{2m-1}$. There are numerous examples of
normal subgroups $N$ for which $\r_{F/N}$ is arbitrarily close to $\r_F$, and
therefore by \eqref{eq:rhov} $v_N$ is arbitrarily close to $\sqrt{2m-1}$. For
instance, if $N_l$ is the kernel of the natural homomorphism
$F=\Z*\Z*\dots*\Z\to \Z_l*\Z_l*\dots*\Z_l$, then $\r_{F/N_l}\to\r_F$ as
$l\to\infty$ by explicit formulas from \cite[Section~9]{Woess00}.
\end{rem}

\begin{rem}
We do not know whether there exist subgroups $H\le F$ whose boundary action is
not completely dissipative (or, even better, is conservative), but
$v_H=\sqrt{2m-1}$. The forthcoming paper \cite{Abert-Glasner-Virag11p}
contains a family of examples of subgroups with $v_H=\sqrt{2m-1}$ (described
in terms of the associated Schreier graphs). However, for all these examples
the action of $H$ is completely dissipative, because the corresponding
Schreier graphs satisfy condition of \propref{pr:RWL}. It is also the case
when the Schreier graph $X$ is radially symmetric. [For, then the radial part
of the simple random walk on $X$ would also have spectral radius $\r_F$ and
would therefore have a positive $\r_F$-invariant function, e.g., see
\cite{Mohar-Woess89} and the references therein. On the other hand, an
explicit calculation shows that if such a function exists, then the number of
cycles in $X$ must be finite, and therefore the action must be completely
dissipative by \thmref{th:pr7}.]
\end{rem}

\begin{rem} \label{rem:qq}
Results similar to \thmref{th:4} are known for Fuchsian groups (where
completely different methods were used). The fact that if the critical
exponent $\d=\d_H$ ($\equiv$ the logarithmic rate of growth with respect to
the Riemannian metric) of a Fuchsian group~$H$ satisfies inequality
$\d<\frac12$, then the boundary action is completely dissipative with respect
to the Lebesgue measure class goes back to Patterson \cite{Patterson77}.
Another proof is given in a recent paper of Matsuzaki \cite{Matsuzaki05},
where examples of Fuchsian groups with conservative boundary action for which
$\d$ is arbitrarily close to $\frac12$ are constructed. However, it is unknown
whether such an example exists with the critical exponent precisely $\frac12$
(cf. \remref{rem:q}).
\end{rem}

\subsection{Conservativity}

Denote by $\g(x)$ the number of edges from the set
$\Edges(X)\setminus \Edges(T)$ incident with a vertex $x\in X$, so
that the degree of $x$ in the spanning tree $T$ is
\begin{equation} \label{eq:deg}
\deg_T(x)=\deg_X(x)-\g(x)=2m-\g(x) \;.
\end{equation}

\begin{prop} \label{pr:5}
\begin{equation} \label{eq:***}
\m\De\ = 1 - \frac{1}{2m}\sum_{x\in X} \frac{\g(x)}{(2m-1)^{|x|}}
\;.
\end{equation}
\end{prop}

\begin{proof}
By \propref{pr:trans} and \thmref{th:1}
$$
\pt F \setminus \De = \bigsqcup_{s\in \bS} C_{\s_-(s)\s_0(s)} \;,
$$
whence by formula \eqref{eq:cylmeas}
$$
\m\De = 1 - \frac1{2m} \sum_{s\in\bS} \frac1{(2m-1)^{|\s_-(s)|}} \;,
$$
which implies the claim in view of the one-to-one correspondence between the
set $\bS$ and the oriented edges not in $T$ established in \thmref{th:span}.
\end{proof}

\begin{thm} \label{th:5}
The sequence
$$
a_n = \frac{|S_X^n|}{|S_F^n|} =
\frac{\left|S_X^n\right|}{2m(2m-1)^{n-1}}
$$
monotonically decreases and converges to $\m\De$.
\end{thm}

\begin{proof}
By \propref{pr:5}
\begin{equation} \label{eq:Dsum}
\m\De = 1 - \frac{1}{2m}\sum_{n=0}^{\infty}\frac{\g_n}{(2m-1)^n} \;,
\end{equation}
where
$$
\g_n = \sum_{x\in S_X^n} \g(x) \;.
$$
In view of \eqref{eq:deg} the numbers $|S_X^n|$ and $\g_n$ are
connected by the formulas
$$
|S_X^1| = \deg_T o = \deg_X o - \g_0 = 2m - \g_0
$$
and
$$
|S_X^{n+1}| = \sum_{x\in S_X^n} (\deg_T x - 1) = (2m-1) |S_X^n| -
\g_n \;, \qquad n>0 \;,
$$
which imply monotonicity of $a_n$. Substituting
$$
\g_0 = 2m - |S_X^1| \;,
$$
and
$$
\g_n = (2m-1) |S_X^n| - |S_X^{n+1}| \;, \qquad n>0
$$
into formula \eqref{eq:Dsum} yields the claim.
\end{proof}

As a corollary we immediately obtain:

\begin{thm} \label{th:cor3}
The boundary action of $H$ is conservative if and only if
$$
\lim_n \frac{|S_X^n|}{|S_F^n|} = \lim_n
\frac{|S_X^n|}{2m(2m-1)^{n-1}} = 0 \;.
$$
\end{thm}

\begin{rem}
In different terms this result is also independently proved in the recent
paper \cite[Theorem 9]{Bahturin-Olshanskii10}.
\end{rem}

\begin{cor} \label{cor:pr5}
If $v_X<2m-1$, then the boundary action of $H$ is conservative.
\end{cor}

\begin{rem}
The converse of \corref{cor:pr5} is not true, see
\exref{ex:maxgrowthcons}.
\end{rem}

\begin{rem} \label{rem:s2}
\thmref{th:cor3} can also be reformulated as saying that the
boundary action is conservative if and only if $|B_X^n|/|B_F^n|\to
0$. In the context of Fuchsian groups this result (obtained by
entirely different means) with $B_F^n$ (resp., $|B_X^n|$) replaced
with the area of the $n$-ball in the hyperbolic plane (resp., in the
quotient surface) was proved by Sullivan
\cite[Theorem~IV]{Sullivan81}, and essentially goes back to Hopf
\cite{Hopf39}.
\end{rem}

\thmref{th:4} and \thmref{th:cor3} imply

\begin{cor}
If $|S_X^n|/|S_F^n|\to 0$ (in particular, if $v_X<2m-1$), then
$v_H\ge\sqrt{2m-1}$.
\end{cor}

\begin{rem}
Since $v_{F/N}<v_F$ for any non-trivial normal subgroup $N\lhd F$ (e.g., see
\cite{Grigorchuk-delaHarpe01}), as another corollary one obtains a
``quantitative'' proof of \thmref{th:normalcons}.
\end{rem}

\subsection{Examples} \label{sec:exx}

\begin{ex}[counterexample to the converse of \corref{cor:6}] \label{ex:mingrowthdiss}
The group $H$ is infinitely generated, the graph $X$ has no hanging
branches, but $v_H$ is arbitrarily close to 1, so that the action of
$H$ on $(\pt F,\m)$ is completely dissipative (see \thmref{th:4}).
\end{ex}

We construct the graph $X$ inductively by starting from the homogeneous tree
$X_0=T_{2m}$ of degree $2m$ with origin $o$. We shall think of the spheres
$S^R_{X_k}$ centered at $o$ in the graphs~$X_k$ as \emph{levels}, and follow
the perverse tradition according to which trees are allowed to grow downwards,
so that the level 0 (consisting just of the origin) is the highest (actually,
in \figref{fig:ex22} below we shall draw it from left to right). In this
construction we shall need an increasing integer sequence $1=d_1<d_2<\dots$ to
be specified later.

The inductive step of the construction consists in choosing two
points $x_{k+1}\neq y_{k+1}\in S^{d_{k+1}}_{X_k}$ for a certain
integer $d_{k+1}$ to be specified later in such a way that their
``predecessors'' in $S^{d_k}_{X_k}$ are distinct. The graph
$X_{k+1}$ is then obtained from $X_k$ in the following way:
\begin{itemize}
    \item
    remove from $X_k$ one of the branches growing from $x_{k+1}$ downwards;
    \item
    do the same with the point $y_{k+1}$;
    \item
    add the edge $\E_{k+1}$ joining $x_{k+1}$ and $y_{k+1}$.
\end{itemize}
Thus, all the graphs $X_k$ are also $2m$-regular and have the same origin $o$.
Finally, the graph $X$ is the limit of the sequence $X_k$, see
\figref{fig:ex22}.

\begin{figure}[h]
\begin{center}
     \psfrag{o}[rt][rt]{$o$}
     \psfrag{s1}[t][t]{$S^{d_1}$}
     \psfrag{s2}[t][t]{$S^{d_2}$}
     \psfrag{s3}[t][t]{$S^{d_3}$}
     \psfrag{s4}[t][t]{$S^{d_4}$}
     \psfrag{x2}[B][B]{$x_2$}
     \psfrag{x3}[B][B]{$x_3$}
     \psfrag{x4}[B][B]{$x_4$}
     \psfrag{y2}[t][t]{$y_2$}
     \psfrag{y3}[t][t]{$y_3$}
     \psfrag{y4}[t][t]{$y_4$}
     \psfrag{E2}[r][r]{$\E_2$}
     \psfrag{E3}[r][r]{$\E_3$}
     \psfrag{E4}[r][r]{$\E_4$}
          \includegraphics[scale=.75]{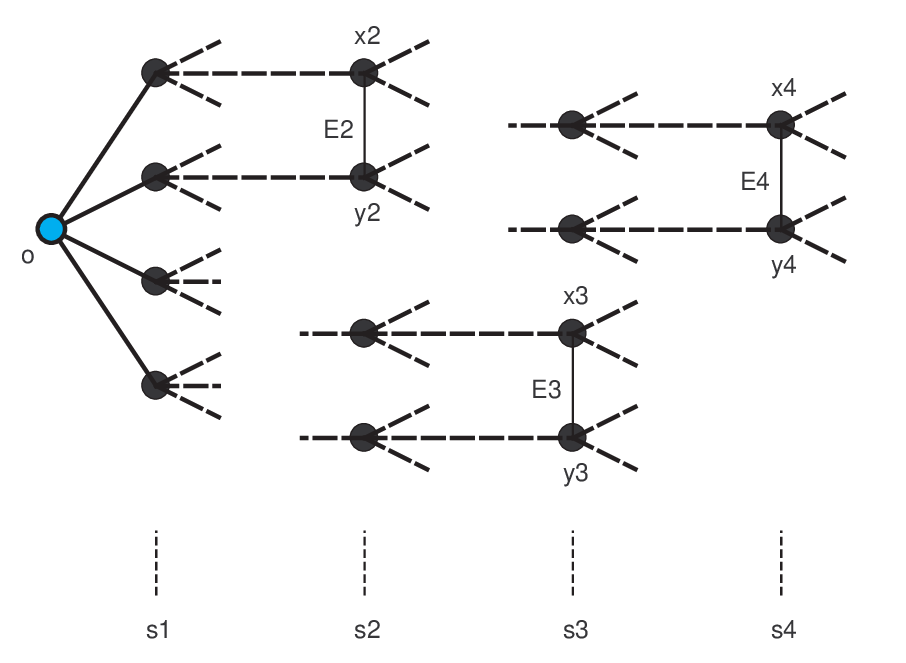}
          \end{center}
          \caption{}
          \label{fig:ex22}
\end{figure}

The graph $X$ has a natural geodesic spanning tree $T$ which is obtained by
removing all the horizontal edges $\E_k$ added during the construction of $X$.
Denote by $s_k$ the corresponding generator of $H$, and let $H_k=\la
s_1,s_2,\dots,s_k\ra$.

Obviously, the sequences of points $x_k,y_k$ can be chosen in such a way that
$X$ has no hanging branches as the latter condition is equivalent to the
property that the intersection of any shadow in the spanning tree $T$ (with
respect to the origin $o$) with the set $\{x_2,y_2,x_3,y_3,\dots\}$ is
non-empty.

\medskip

We shall now explain (once again inductively) how to choose the
sequence $d_k$. More precisely, we shall show that once the numbers
$d_1,d_2,\dots,d_k$ and the group $H_k$ have already been chosen,
and the group $H_k$ has the property that
\begin{equation} \label{eq:gr}
|H_k \cap B_F^n| \le C \a_k^n \qquad\forall\, n\ge 0
\end{equation}
for certain constants $C,\a_k>1$, then for any $\a_{k+1}>\a_k$ the
distance $d_{k+1}$ can be chosen in such a way that for the group
$H_{k+1}$ also
\begin{equation} \label{eq:grr}
|H_{k+1} \cap B_F^n| \le C \a_{k+1}^n \qquad\forall\, n\ge 0
\end{equation}
(with the same constant $C$!). Then, starting from the group
$H_1\cong\Z$ (which has subexponential growth), and taking an
arbitrary sequence
$$
1<\a_1<\a_2<\dots<\a_k\dots\nearrow\a
$$
we would be able to conclude that $v_H\le\a$.

\medskip

For notational simplicity we put $s=s_{k+1}$. Any element of the
group $H_{k+1}=\la H_k,s\ra$ can be presented as
\begin{equation} \label{eq:prod}
g = h_0 s^{\e_1} h_1 s^{\e_2} \dots h_{t-1} s^{\e_t} h_t
\end{equation}
for certain $t\ge 0,\, h_i\in H$ and $\e_i=\pm 1$ such that
$\e_i=\e_{i+1}$ whenever $h_i=e$. As it follows from the definition
of $s$, the length of cancellations on each side between any two
consecutive terms in the above expansion does not exceed $d_k$.
Since $|s|=2d_{k+1}+1$, we obtain the inequality
\begin{equation} \label{eq:est}
|g| \ge \sum_{i=0}^t |h_i| + t|s| - 4t d_k = \sum_{i=0}^t |h_i| + tD
\;,
\end{equation}
where
\begin{equation} \label{eq:D}
D=|s|-4d_k = 2 d_{k+1} + 1 - 4 d_k \;.
\end{equation}

\medskip

We shall now estimate $|H_{k+1}\cap B_F^n|$, i.e., the number of
elements $g$ of the form \eqref{eq:prod} with $|g|\le n$, by using
the inequality \eqref{eq:est}. We have to control the following
numbers:
\begin{itemize}
    \item[(a)]
    The number $t$ of occurrences of $s^{\pm 1}$ in the expansion
    \eqref{eq:prod};
    \item[(b)]
    The number $N_b$ of choices of the signs $\e_i$ for a given value of $t$;
    \item[(c)]
    The number $N_c$ of possible sets of lengths $|h_i|=l_i$ of the words $h_i$ for
    given $t$;
    \item[(d)]
    The number $N_d$ of the choices of the words $h_i\in H_k$ with the
    prescribed lengths $l_i$.
\end{itemize}
Let us find the corresponding estimates one by one.

\medskip

(a) By \eqref{eq:est}
$$
t \le \t = n/D \;.
$$

\medskip

(b) Trivially,
$$
N_b \le 2^\t \;.
$$

\medskip

(c) By \eqref{eq:est}, $N_c$ does not exceed the number of ordered
partitions of $n-tD$ into not more than $t+1$ integer summands, so
that
$$
N_c \le \binom{n-tD+t}{t} \;.
$$

\medskip

(d) Finally, by \eqref{eq:gr} and \eqref{eq:est}
$$
N_d \le C^\t \a_k^n \;.
$$

Thus,
$$
|H_{k+1}\cap B_F^n| \le \t 2^\t  \max_{t\le\t} \binom{n-tD+t}{t}C^\t
\a_k^n \;,
$$
and in order to conclude it is sufficient just to estimate the
$\max$ in the above product. Since
$$
\begin{aligned}
  &\frac1n \log\binom{n-tD+t}{t}
 = \frac1n \log\binom{n-tD'}{t}\\
 &\le \frac1n \log\frac{(n-tD')^t}{(t/e)^t}
 =\frac{t}{n} \left( 1 + \log\left( \frac{n}{t} - D' \right) \right)
 = \frac{1+\log(z-D')}{z} \;,
\end{aligned}
$$
where $D'=D-1$ and $z=n/t\ge D$, we can choose $D$ (equivalently,
$d_{k+1}$, in view of formula \eqref{eq:D}) in such a way that
$$
\sup_{z\ge D} \frac{1+\log(z-D')}{z}
$$
is arbitrarily small, whence the claim.

\begin{ex}[counterexample to the converse of \thmref{th:4}] \label{ex:amendiss}
The cogrowth $v_H=2m-1$ is maximal, i.e., the Schreier graph $X$ is
amenable (see \secref{sec:grcogr}), but the boundary action is
completely dissipative.
\end{ex}

Take the homogeneous tree $T_{2m}$ with the root $o$, remove one of the
branches rooted at~$o$, and replace it with a geodesic ray $\r\cong\Z_+$. Then
attach to all the vertices of $\r$ other than the origin $m-1$ length 1 loops,
so that the resulting graph $X$ is $2m$-regular: it is a union of $2m-1$
hanging branches and the ray $\r$ (with attached loops) joined at the point
$o$, see \figref{fig:ex28}. The graph $X$ is obviously amenable because of the
presence of the~$\r$ branch. On the other hand, the simple random walk on it
eventually stays inside one of the hanging branches (since the simple random
walk on $\Z_+$ is recurrent), whence the claim in view of \corref{cor:RWD}.

\begin{figure}[h]
\begin{center}
     \psfrag{o}[rb][rb]{$o$}
          \includegraphics[scale=.75]{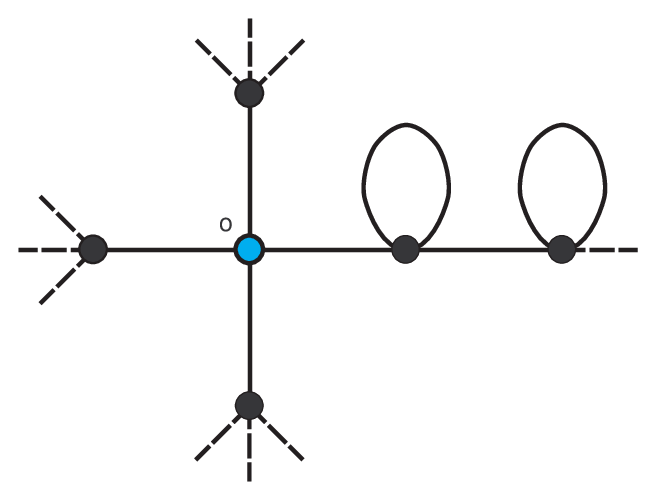}
          \end{center}
          \caption{}
          \label{fig:ex28}
\end{figure}

\begin{ex}[counterexample to the converse of \corref{cor:pr5}] \label{ex:maxgrowthcons}
The growth $v_X=2m-1$ is maximal, but the boundary action is
conservative.
\end{ex}

Take the homogeneous tree $T_{2m}$ with the root $o$ and take an increasing
sequence $d_n$ with density 0, i.e., such that $d_n/n\to\infty$. Then add one
new vertex in the middle of each edge of $T_{2m}$ joining the spheres of radii
$d_n$ and $d_n+1$ and attach to every such vertex $m-1$ length 1 loops. Then
the resulting graph $X$ is radially symmetric, $2m$-regular, and has the
growth $v_X=2m-1$, although it satisfies the condition of \thmref{th:cor3}.
The radial part of this graph is presented on \figref{fig:ex29}.

\begin{figure}[h]
\begin{center}
     \psfrag{o}[rt][rt]{$o$}
     \psfrag{d}[rb][rb]{$d_n$}
     \psfrag{e}[lb][lb]{$d_n+1$}
          \includegraphics[scale=.75]{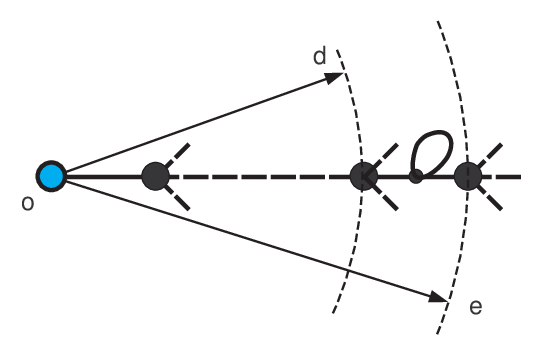}
          \end{center}
          \caption{}
          \label{fig:ex29}
\end{figure}

\begin{ex}[counterexample to an extension of \thmref{th:pr7}] \label{ex:consdiss}
An infinitely generated subgroup $H$ such that its boundary action
has both conservative and dissipative parts of positive measure.
\end{ex}

Take a group $\wt H\le F$ such that the simple random walk on the associated
Schreier graph $\wt X$ is transient and the boundary action of $\wt H$ is
conservative (for instance, this is the case when $\wt H$ is a normal subgroup
with transient quotient, see \thmref{th:normalcons}). Then necessarily $\wt H$
is infinitely generated by \thmref{th:pr7}. Let $\wt\S$ be the system of
generators of $\wt H$ determined by a geodesic spanning tree $\wt T$. Then for
any $s\in\wt\S$ the group $H=\la\S\ra,\,\S=\wt\S\setminus\{s\}$ has the
desired property.

Indeed, by \lemref{lem:smaller} the Schreier graph $X=H\bs F$ has hanging
branches, which implies that the dissipative part of the boundary action of
$H$ is non-trivial (see \corref{cor:6}).

Let us now prove non-triviality of the conservative part. Let
$\E=\E_s\in\Edges(\wt X)\setminus\Edges(\wt T)$ be the edge corresponding to
the generator $s$, and let $A\subset\pt F$ be the set of all points $\o$ such
that the associated path $\wt\pi(\o)$ in $\wt X$ never passes through the edge
$\E$ (in either direction). Since the boundary action of $\wt H$ is
conservative, by \thmref{th:2} for a.e.\ point $\o\in A$ the path $\wt\pi(\o)$
passes nonetheless through infinitely many edges not in~$\wt T$.

Transience of the simple random walk implies that $\m A>0$. Namely, by using
labelling along edges every trajectory $(x_n)$ of the random walk on $X$ lifts
to a trajectory $(g_n)$ of the simple random walk on $F$. Denote by
$g_\infty\in\pt F$ its limit point. Then the set of edges through which the
path $\wt\pi(g_\infty)$ passes is contained in the analogous set for the
original sample path $(x_n)$. Thus, if $(x_n)$ never passes through $\E$, then
the path $\wt\pi(g_\infty)$ also has this property. By transience the
probability of the former event is positive, and since the image of the
measure in the space of sample paths under the above transformation is $\m$,
the claim follows.

Now, for any $\o\in A$ the associated path $\pi(\o)$ in the Schreier graph $X$
passes through the same edges as the path $\wt\pi(\o)$ in $\wt X$ (under the
natural identification described in \lemref{lem:smaller}), i.e., it passes
through the edges which are not in the spanning tree $T$ of $X$ infinitely
many times. Since $\m A>0$, it means that the conservative part of the
boundary action of $H$ is also non-trivial.

\begin{rem} \label{rem:Liouv}
One can also show that in the above construction the ergodic components of the
boundary action of $\wt H$ are in one-to-one correspondence with the ergodic
components of the conservative part of the boundary action of $H$. In
particular, if the boundary action of $\wt H$ is ergodic, then the
conservative part of the boundary action of $H$ is also ergodic. It follows
from the fact that the Poisson boundary of the simple random walk on $\wt X$
does not change after removing the edge $\E$.
\end{rem}

\begin{rem}
In the context of covering Markov operators a similar example with both
conservative and dissipative components in the Poisson boundary was
constructed in \cite{Kaimanovich95}.
\end{rem}

\begin{rem}
In the above example the measure $\m$ of the full limit set $\La$ is
intermediate between 0 and 1. Indeed, $\m\La<1$ because the Schreier graph has
a hanging branch. On the other hand, $\m\La>0$ because the boundary action has
a non-trivial conservative part.
\end{rem}

\section{Associated Markov chains} \label{sec:asso-chains}

As we have already seen, the simple random walk on the Schreier graph plays
the crucial role in understanding the ergodic properties of the boundary
action. In this Section we shall look at two other Markov chains closely
connected with the considered problems.

\subsection{Random walk on edges}

The uniform boundary measure $\m$ can be interpreted as the measure in the
path space of the following Markov chain $(\E_n)$ on the set of oriented edges
of the Schreier graph $X$. Its initial distribution is uniform on the set of
$2m$ edges issued from the origin $o$, and the transition probability from an
arbitrary edge $\E$ is equidistributed on the set of $2m-1$ edges $\E'$ issued
from the endpoint of $\E$ without backtracking (in other words, the labels on
$\E$ and $\E'$ do not cancel). Then \propref{pr:2} and \thmref{th:2} imply

\begin{thm} \label{th:markovedges}
The boundary action of $H$ is conservative (resp., dissipative) if a.e.\
sample path of the chain $\E_n$ visits the set $\Edges(X)\setminus\Edges(T)$
infinitely often (resp., finitely many times).
\end{thm}

\begin{rem}
Deciding whether a given set is visited or not with positive probability by
sample paths of a certain Markov chain is a classical problem in probability
(and potential theory) which goes back to Kakutani. Explicit estimates for the
probabilities of visiting the set infinitely many times or not visiting it at
all are given in terms of various kinds of \emph{capacity}, see
\cite{Benjamini-Pemantle-Peres95}.
\end{rem}

\subsection{Markov chain on cycles}

We shall now assume that the conservative part of the action of $H$
on $(\pt F,\m)$ is non-trivial, i.e., $\m\O>0$. We give its symbolic
interpretation and show that its ergodicity is equivalent to the
Liouville property for a certain naturally associated Markov chain.

Let us endow $\pt H$ with the probability measure $\m_*$ which is the preimage
of the normalized restriction $\m|_\O/\m\O$ with respect to the map
$\s^\infty$ (see \thmref{th:bdrymap}). Obviously, the dynamical systems
$(H,\O,\m|_\O)$ and $(H,\pt H,\m_*)$ are isomorphic (up to the constant
multiplier used to normalize the measure $\m|_\O$). In particular, the action
of $H$ on the space $(\pt H,\m_*)$ is conservative.

Recall that the space $\pt H$, being the set of infinite irreducible words in
the alphabet~$\bS$, is the state space of a topological Markov chain. The
alphabet $\bS$ of this chain is in one-to-one correspondence with a set of
cycles in the Schreier graph $X$ (see \secref{sec:Sch}).

\begin{thm} \label{th:pr10}
The measure $\m_*$ on $\pt H$ is Markov in the alphabet $\bS$. It
corresponds to the initial distribution
$$
\th(s) = \m_* C_s \;, \qquad s\in\bS \;,
$$
and the transition matrix
\begin{equation} \label{eq:M}
M(s,s') = (2m-1)^{|s'|-|ss'|} \cdot \frac{\th(s')}{\th(s)} \;.
\end{equation}
\end{thm}

\begin{proof}
The argument basically consists in noticing that, due to the special
properties of the generating set $\S$, the Radon--Nikodym derivative
\begin{equation} \label{eq:*a}
\frac{d\m_*(s_2s_3\dots)}{d\m_*(s_1s_2s_3\dots)} =
\frac{ds_1\m_*}{d\m_*}(\xi) \;, \quad \xi=s_1s_2s_3\dots\in\pt H \;,
\end{equation}
of the measure $\m_*$ is \emph{Markov} in the sense that it depends
on the letters $s_1,s_2$ only. [In the language of symbolic dynamics
this derivative, or, more rigorously, its logarithm, is called the
\emph{potential} of the measure $\m$.] Indeed, by \propref{pr:RN}
and \propref{pr:RNOmega}
$$
\frac{d\m_*(s_2s_3\dots)}{d\m_*(s_1s_2s_3\dots)} =
(2m-1)^{|s_1s_2|-|s_2|} \;,
$$
whence for an arbitrary cylinder set $C_{s_1s_2\dots s_n}$
$$
\frac{\m_* C_{s_2\dots s_n}}{\m_* C_{s_1s_2\dots s_n}} =
(2m-1)^{|s_1s_2|-|s_2|} \;.
$$
Comparing the formula
$$
\begin{aligned}
\m_* C_{s_1s_2\dots s_n} &= \frac{\m_* C_{s_1s_2\dots s_n}}{\m_*
C_{s_2\dots s_n}} \cdot \frac{\m_* C_{s_2\dots s_n}}{\m_*
C_{s_3\dots s_n}} \cdot \dots \cdot \frac{\m_* C_{s_{n-1}s_n}}{\m_*
C_{s_n}} \cdot \m_* C_{s_n} \\
&= (2m-1)^{(|s_2|-|s_1s_2|)+(|s_3|-|s_2s_3|) + \dots
+(|s_n|-|s_{n-1}s_n|)} \cdot \th(s_n)
\end{aligned}
$$
with the analogous expansion for $\m_*C_{s_1s_2\dots s_ns_{n+1}}$ we
get the claim.
\end{proof}

\subsection{Applications to the boundary action}

\begin{thm} \label{th:7}
The following measure spaces are canonically isomorphic:
\begin{itemize}
    \item[(i)]
    The Poisson boundary of the Markov chain on $\bS$ described in
    \thmref{th:pr10};
    \item[(ii)]
    The space of ergodic components of the (one-sided) time shift in the measure
    space $(\pt H,\m_*)$;
    \item[(iii)]
    The space of ergodic components of the action of the group $H$
    on the measure space $(\pt H,\m_*)$.
\end{itemize}
\end{thm}

\begin{proof}
The isomorphism of the spaces (i) and (ii) is a general fact from the theory
of Markov chains, e.g., see \cite[Definition~1.5]{Kaimanovich92}, whereas the
isomorphism of the spaces (ii) and (iii) follows from the coincidence of the
orbit equivalence relations of the time shift and of the $H$-action on $\pt
H$.
\end{proof}

\begin{rem}
Of course, \thmref{th:7} remains valid for an arbitrary Markov measure on~$\pt
H$ with the full support.
\end{rem}

\begin{cor} \label{cor:5}
The Schreier graph $X$ is Liouville if and only if the boundary action of~$H$
on $(\pt F,\m)$ is conservative and the Markov chain described in
\thmref{th:pr10} is Liouville.
\end{cor}

\begin{rem}
There are examples when the Markov chain described in
\thmref{th:pr10} is Liouville although the Schreier graph $X$ is
not. They correspond to the situation when the boundary action of
$H$ on $(\pt F,\m)$ has both conservative and dissipative parts, and
the conservative part is ergodic, see \exref{ex:consdiss} and
\remref{rem:Liouv}.
\end{rem}

\begin{rem}
Yet another example of a Markov measure on $\pt H$ is provided by the harmonic
(hitting) measure of a random walk $(H,\mu)$ with $\supp\mu=\bS$. It was known
already to Dynkin and Malyutov \cite{Dynkin-Malutov61} (also see
\cite{Levit-Molchanov71}), that if $H$ is finitely generated, then the sample
paths converge a.e.\ to the boundary $\pt H$, the hitting measure $\l$ is
Markov, and the space $(\pt H,\l)$ is the Poisson boundary of the random walk
$(H,\mu)$. A recent addition to these facts is the observation that the
hitting measure $\l$ is in fact \emph{multiplicative Markov}, and not just
plain Markov \cite{Mairesse05, Mairesse-Matheus07}, i.e., the transition
probability from any $s\in\bS$ is the normalized restriction of the initial
distribution onto the set $\bS\setminus\{s^{-1}\}$ of all letters admissible
from $s$. These results readily generalize to the situation when $H$ is
infinitely generated by using the approximation of $H=\la s_1,s_2,\dots\ra$ by
finitely generated subgroups $H_n=\la s_1,s_2,\dots,s_n\ra$. In particular,
the hitting measure $\l$ on $\pt H$ is just the projective limit of the
hitting measures $\l_n$ on $\pt H_n$; since each of the measures $\l_n$ is
multiplicative Markov, the limit measure $\l$ is also multiplicative Markov.
This fact can be used to show that, generally speaking, the measure $\m_*$ is
not equivalent to the hitting measure of a random walk on $H$. Indeed, the
Markov measure with the transition probabilities \eqref{eq:M} is
multiplicative only if for any $s'\neq s''\in\bS$ the difference
$|ss'|-|ss''|$ is the same for all $s\in\bS$ with $s',s''\neq s^{-1}$, which
is a quite restrictive condition.
\end{rem}

\bibliographystyle{amsalpha}
\bibliography{C:/Sorted/MyTex/mine}

\enddocument

\bye